\documentclass{article}
\usepackage[utf8]{inputenc}
\usepackage{amsmath}
\usepackage{amssymb}
\usepackage{mathrsfs}
\usepackage{amsfonts}
\usepackage{amsthm}
\usepackage{bm}
\usepackage{comment}
\usepackage{float}
\usepackage{booktabs}
\usepackage{tikz}
\usepackage{multirow}
\usepackage{hhline}
\usepackage{makecell}
\usepackage{caption}
\usepackage{todonotes}
\usepackage{xcolor}

\usetikzlibrary{matrix,decorations.pathmorphing, positioning}
\usetikzlibrary{arrows,automata,backgrounds,petri}
\usetikzlibrary{shapes.multipart}
\usetikzlibrary{decorations.pathreplacing}
\usetikzlibrary{calc, trees}

\newenvironment{psmallmatrix}
  {\left(\begin{smallmatrix}}
  {\end{smallmatrix}\right)}
  
\usepackage[alphabetic, initials]{amsrefs}

\DeclareMathOperator{\Ht}{Ht}

\theoremstyle{definition}
\newtheorem{definition}{Definition}[section]
\newtheorem{thm}{Theorem}[section]

\newtheorem{prop}[thm]{Proposition}
\newtheorem{lem}[thm]{Lemma}

\newtheorem{rem}[thm]{Remark}

\providecommand{\subjclass}[1]
{
  \small	
  \textbf{\textit{Mathematics Subject Classification (2020) }} #1
}

\providecommand{\keywords}[1]
{
  \small	
  \textbf{\textit{Keywords }} #1
}

\title{On the Markoff spectrum on the Hecke group of index six}

\author{Byungchul Cha \and Dong Han Kim \and Deokwon Sim}
\newcommand{\Addresses}{{
\bigskip
\footnotesize
  B. Cha, \textsc{Department of Mathematics Education,
  Hongik University, 94 Wausan-ro, Mapo-gu, Seoul 04066, Korea}\par\nopagebreak
  \textit{E-mail address}: \texttt{cha@hongik.ac.kr}

  \medskip

  D.H. Kim, \textsc{Department of Mathematics Education, Dongguk University - Seoul, 30 Pildong-ro 1-gil, Jung-gu, Seoul, 04620 Korea}\par\nopagebreak
  \textit{E-mail address}: \texttt{kim2010@dgu.ac.kr}

  \medskip

  D. Sim, \textsc{Machine Learning TU, Samsung Advanced Institute of Technology, 130 Samsung-ro, Yeongtong-gu, Suwon-si, Gyeonggi-do, 16678 Korea}\par\nopagebreak
  \textit{E-mail address}: \texttt{deokwon.sim@snu.ac.kr}

}
}

\date{\today}
\begin{document}

\maketitle

\begin{abstract}
The discrete part of the Markoff spectrum on the Hecke group of index 6 was determined by A.~Schmidt. 
In this paper, we study its Markoff and Lagrange spectra after the smallest accumulation point $4/\sqrt3$.
We show that both the Markoff and Lagrange spectra below $4/\sqrt{3} + \epsilon$ have positive Hausdorff dimension for any positive $\epsilon$. 
We also find maximal gaps and an isolated point in the spectra. 
\end{abstract}

\medskip


\medskip


\section{Introduction}

Consider a finitely generated Fuchsian group $\mathbf{G}$ acting on 
the upper half plane 
$\mathbb{H}$. 
We define the notion of the Lagrange spectrum of $\mathbf{G}$ as follows (cf.~\cite{Leh85}, \cite{HS86}).
Let $\mathbb{Q}(\mathbf{G})$ be the set of points in $\mathbb{R} \cup \{ \infty \}$ that are fixed by parabolic elements of $\mathbf{G}$. 
We shall assume that $\infty\in \mathbb{Q}(\mathbf{G})$. 
Write $g\in\mathbf{G}$ as 
$g = 
\begin{psmallmatrix}
    a &  b \\ c & d
\end{psmallmatrix}
\in \mathrm{PSL}_2(\mathbb{R})$, 
so that the action of $g$ on $\mathbb{H}$ is expressed by linear fractional transformation 
\[
g(z) = \frac{a z + b}{c z + d}.
\]
To emphasize the dependency on $g$, we write $a = a(g)$ and $c = c(g)$. 
For $\xi\not\in\mathbb{Q}(\mathbf{G})$, we define its \emph{Lagrange number} $L_{\mathbf{G}}(\xi)$ to be
\begin{equation}\label{def_LG}
L_\mathbf{G}(\xi) = 
\limsup_{g \in \mathbf{G}} \left( c(g)^2 \left| \xi - g(\infty) \right| \right)^{-1}.
\end{equation}
%\begin{chasuggestion}
%Are we talking about a geodesic in the modular surface $\mathbb{H}/\mathbf{G}$ or a \emph{lift} of the geodesic in $\mathbb{H}$?
%\end{chasuggestion}
In fact, $L_\mathbf{G}(\xi)$ is the limit superior of the height of the geodesic from $\infty$ to $\xi$ in the modular surface $\mathbb{H}/\mathbf{G}$ (See Section~\ref{Sec:Romik} below).
See Figure~\ref{Fig:geodesics} for the three geodesics of smallest heights on $\mathbb{H}/\mathbf H_6$.
The
\emph{Lagrange spectrum of} $\mathbf{G}$ is defined to be
\[
\mathscr L(\mathbf{G}) : = \left\{ L_\mathbf{G} (\xi) 
\, \Big| \, \xi \in \mathbb R \setminus \mathbb Q(\mathbf{G}) \right\}.
\]

\begin{figure}
\centering
\begin{tikzpicture}[scale=2].  % 11111
%	\node[below] at (-1/2, 0) {$-\frac{1}2$};
%	\node[below] at (0, 0) {$\frac 01$};
%	\node[below] at (1/2, 0) {$\frac 12$};
\node[below] at (0, 1) {$i$};

\node at (0, 1.7) {$F$};
	
%	\draw[thick] (-1.2, 0) -- (3.2, 0);
\draw (-.8660,1.8) -- (-.8660,.5);
\draw (.8660,1.8) -- (.8660,.5);
\draw (.8660,.5) arc (30:150:1);

\draw[very thick] (.8660,.5) arc (30:150:1);
\end{tikzpicture}
\
\begin{tikzpicture}[scale=2]. 
\node[below] at (0, 1) {$i$};

\node at (0, 1.7) {$F$};
	
%	\draw[thick] (-1.2, 0) -- (3.2, 0);
\draw (-.8660,1.8) -- (-.8660,.5);
\draw (.8660,1.8) -- (.8660,.5);
\draw (.8660,.5) arc (30:150:1);

\draw[very thick] (.8660,.8660) arc (56.309932:106.102113:1.0408);
\draw[very thick] (-.8660,.8660) arc (56.309932:43.85377:1.0408);
\draw[very thick] (-.8660,.8660) arc (123.69006:73.897886:1.0408);
\draw[very thick] (.8660,.8660) arc (123.69006:136.1462213:1.0408);
\end{tikzpicture}
\
\begin{tikzpicture}[scale=2].
\node[below] at (0, 1) {$i$};

\node at (0, 1.7) {$F$};
\draw (-.8660,1.8) -- (-.8660,.5);
\draw (.8660,1.8) -- (.8660,.5);
\draw (.8660,.5) arc (30:150:1);

\draw[very thick] (-.8660,.5) arc (154.30661909:53.076267:1.1532562);
\draw[very thick] (-.8660,0.92195444) arc (53.076267:39.858746:1.1532562);
\draw[very thick] (.8660,.891882) arc (121.7125152:135.1715213:1.048414);
\draw[very thick] (-.8660,.891882) arc (121.7125152:72.5198297:1.048414);
\draw[very thick] (.8660,.5) arc ({180-154.30661909}:{180-53.076267}:1.1532562);
\draw[very thick] (.8660,0.92195444) arc ({180-53.076267}:{180-39.858746}:1.1532562);
\draw[very thick] (-.8660,.891882) arc ({180-121.7125152}:{180-135.1715213}:1.048414);
\draw[very thick] (.8660,.891882) arc ({180-121.7125152}:{180-72.5198297}:1.048414);

\end{tikzpicture}
\caption{Three closed geodesics in 
%the fundamental domain of group $\mathbf{H}_6$ 
%on the upper half space 
%{\color{blue} or should we say 
\emph{the modular surface} $\mathbb{H}/\mathbf{H}_6$ 
with lowest heights.
%($2$ - $[\,\overline{2}\,]$, $\sqrt 6$- $[\,\overline{3,1}\,]$, $\frac{2\sqrt{17}}{3}$- $[\,\overline{3,1,2}\,]$)
}
\label{Fig:geodesics}
\end{figure}
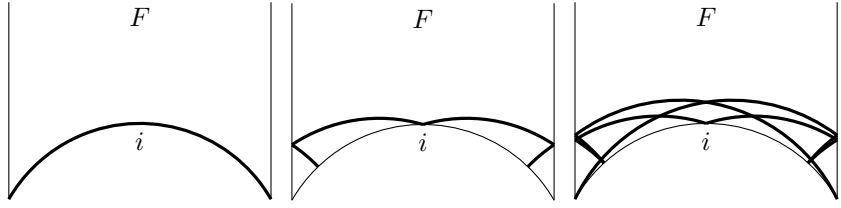

To define the Markoff spectrum of $\mathbf{G}$ (cf.~\cite{Sch76}), we 
consider a (real-coefficient) binary quadratic form 
$f(x, y)= \alpha x^2 + \beta xy + \gamma y^2$ 
with discriminant $\Delta(f) = \beta^2 - 4\alpha\gamma$ 
and write, for $g\in\mathbf{G}$, 
\begin{equation}\label{def_fg}
f(g) := 
f(
g 
\begin{psmallmatrix}
    1 \\ 0
\end{psmallmatrix}
) =f(a(g), c(g)).
\end{equation}
When $\Delta(f) > 0$, we define
\[
m_\mathbf{G}(f) 
=
\sup_{g\in \mathbf{G}}
\frac{\sqrt{\Delta(f)}}{|f(g)|}.
\]
The \emph{Markoff spectrum of} $\mathbf{G}$ is defined by
\[
\mathscr M(\mathbf{G})  = \left\{ m_\mathbf{G} (f) 
\, \Big| \, \Delta(f) > 0 \right\}.
\]

In this paper, we study the structure of the Lagrange and Markoff spectra of $\mathbf{G}$ when $\mathbf{G}$ is the \emph{index $q$ Hecke group} $\mathbf{H}_q$ ($q\ge3$), that is, the subgroup of $\mathrm{PSL}_2(\mathbb{R})$ generated by
\[
S = \begin{pmatrix} 0 & -1 \\ 1 & 0 \end{pmatrix} 
\qquad
\text{and}
\qquad
T = \begin{pmatrix} 1 & \lambda_q \\ 0 & 1 \end{pmatrix}
\]
where $\lambda_q = 2 \cos \frac{\pi}{q}$.
In particular, we are interested in studying the structure of 
$\mathscr{L}(\mathbf{H}_q)$
and
$\mathscr{M}(\mathbf{H}_q)$ when $q = 6$.
For $q = 5$, the spectrum was investigated by Series \cite{Ser88} and Vulakh studied the case of general $q \ge 6$ in \cite{Vul97}.
Recently, two of the authors of the present article investigated the case $q = 4$ in \cite{KS22}.

The Lagrange and Markoff spectra of several Hecke groups turn out to be closely related to some other spectra arising from different contexts. 
For example, when $q = 3$, it is easy to see that $\mathbf{H}_3 = \mathrm{PSL}_2(\mathbb{Z})$ and therefore
$\mathscr{L}(\mathbf{H}_3)$ and 
$\mathscr{M}(\mathbf{H}_3)$ are nothing but 
the classical Lagrange and Markoff spectra. 
In \cite{KS22}, it is explained that 
the spectrum for $\mathbf{H}_4$ is the same as the so-called 2-spectrum of A.~Schmidt \cite{Sch76}, as well as the spectrum arising from intrinsic Diophantine approximation of a unit circle $S^1$ in $\mathbb{R}^2$ \cite{CK22}. 

The case $q = 6$, which is the subject of the present article, is related to the 3-spectrum of A.~Schmidt \cite{Sch77}.
He considered a Fuchsian group $\Phi_3$, which is  
conjugate to $\mathbf{H}_6$.
Additionally, Schmidt's 3-spectrum is the same as the spectrum arising from the intrinsic Diophantine approximation associated to the curve defined by $x^2 + xy + y^2 = 1$ in $\mathbb{R}^2$, as is explained in \cite{CCGW22}.

The spectra we considered above, as well as some others, %and some others we didn't discuss here, 
share certain structural similarities. 
First, they begin with a discrete part consisting of at most countably many points up to the %its
smallest accumulation point. 
Indeed, Vulakh used some geometric methods to describe the discrete part of a spectrum that arises from a certain class of Fuchsian groups, including the Hecke groups (see \cite{Vul97} and \cite{Vul00}).
In the other extreme, the spectrum contains a half interval leading up to $\infty$, which is often called \emph{Hall's ray}.
Artigiani, Marchese, and Ulcigrai recently established in \cite{AMU20} the existence of Hall's ray in a very general setting. 
The middle part of a spectrum, that is, the region between its smallest accumulation point and the start of its Hall's ray appears to be the most technically difficult part to describe.
The purpose of the present paper is to study the middle part of 
$\mathscr{M}(\mathbf{H}_6)$ 
utilizing the digit sequences
developed in our prior work \cite{CK22}, \cite{KS22}, \cite{CCGW22}.

Concerning their initial discrete parts, the spectra for 
$\mathbf{H}_3$ and
$\mathbf{H}_4$
share a structural similarity;
the Lagrange numbers in the initial discrete parts in both cases are governed by the so-called \emph{Fricke identity} 
\[
x_1^2 + x_2^2 + x_3^2 = x_1x_2x_3.
\]
For example, see Proposition 4.3 and the equation (4.1) in \cite{Aig13}.
On the other hand, 
for the case of $\mathbf{H}_6$, which is the only arithmetic group other than $\mathbf{H}_3$ and $\mathbf{H}_4$ among all Hecke groups $\mathbf{H}_q$, 
the Lagrange numbers in the initial discrete part 
are not governed by the Fricke identity and
%{\color{red} under the smallest accumulation point $\frac4{\sqrt3}$}
have a simpler structure.
See \cite{CCGW22} and \cite{Sch76} for detail.
However, we show in the present paper that 
the middle part of 
$\mathscr{M}(\mathbf{H}_6)$ 
shares similar features with the classical Markoff spectrum.
Recall that, in the classical case (see \cite{Cus74}),
there are several maximal gaps of the spectrum in its middle part, the largest of which is the interval $\left(\sqrt{12},\sqrt{13}\right)$, and an isolated point $\sqrt{13}$.
Recently, Moreira \cite{Mor18} showed that the (classical) Lagrange spectrum has positive Hausdorff dimension past its smallest accumulation point and has full Hausdorff dimension below $\sqrt{12}$.

Regarding $\mathbf{H}_6$, we prove the following theorems.
%{\color{red} 
Here and after, we denote by $\dim_H$ the Hausdorff dimension.
\begin{thm}\label{thm_dimension}
For any $\varepsilon >0$, we have
\[
\dim_H \left( \mathscr M (\mathbf H_6) \cap \left[0, \frac{4}{\sqrt{3}} + \epsilon\right) \right) \ge \dim_H \left( \mathscr L (\mathbf H_6) \cap \left[0, \frac{4}{\sqrt{3}} + \epsilon\right) \right) > 0.
\]
\end{thm}
\begin{thm}\label{thm1}
The intervals 
$\left(\frac{\sqrt{143}}{5}, \sqrt{7} \right)$ and $\left( \sqrt{7}, \frac{13\sqrt{3}+13\sqrt{7}+\sqrt{143}}{26} \right)$
are maximal gaps in $\mathscr{M}(\mathbf{H}_6)$ and $\mathscr{L}(\mathbf{H}_6)$.
In particular, $\sqrt7$ is an isolated point of $\mathscr{M}(\mathbf{H}_6)$ and $\mathscr{L}(\mathbf{H}_6)$.
\end{thm}
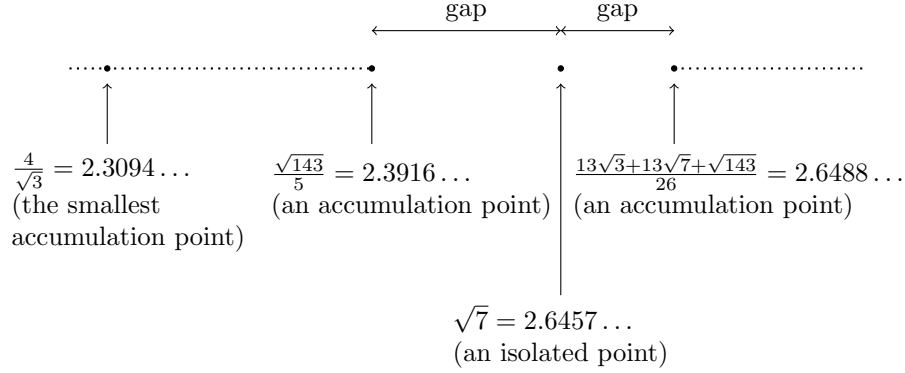
\begin{figure}
\begin{center}
\begin{tikzpicture}[every text node part/.style={align=left}]
  \draw[thick, dotted] (1, 0) -- (5, 0);
%  \draw[thick, dotted] (3.5, 0) -- (5, 0);
  \draw[thick, dotted] (9, 0) -- (11.5, 0);
  \filldraw (1.5,0) circle (1pt);
%  \filldraw (3.5,0) circle (1pt);
  \filldraw (5,0) circle (1pt);
  \filldraw (7.5,0) circle (1pt);
  \filldraw (9,0) circle (1pt);
  
%  \draw[thin, <->] (1.5, 0.5) -- node[midway, above]{gap} (3.5, 0.5);
  \draw[thin, <->] (5, 0.5) -- node[midway, above]{gap} (7.5, 0.5);
  \draw[thin, <->] (7.5, 0.5) -- node[midway, above]{gap} (9, 0.5);

%  \draw[thin, ->] (3.5, -3) node[below] {%
%    $\frac{2\sqrt{2803333}}{1405} =2.3833\dots$ 
%    \\ 
%    $=\mathcal{M}({}^{\infty}(4323243)^{\infty})%
%  $} -- (3.5, -0.2); 
  \draw[thin, ->] (1.5, -1) node[below, xshift = 8pt] {%
    $\frac{4}{\sqrt3} 
    = 2.3094\dots$ \\
    (the smallest \\ accumulation point)
} -- (1.5, -0.2);

  \draw[thin, ->] (7.5, -3) node[below] {%
    $\sqrt{7} =2.6457\dots$ 
    \\ 
    (an isolated point)
} -- (7.5, -0.2); 
  \draw[thin, ->] (5, -1) node[below, xshift=15pt] {%
    $\frac{\sqrt{143}}{5} 
    = 2.3916\dots$ \\
  (an accumulation point)
} -- (5, -0.2);

\draw[thin, ->] (9, -1) node[below, xshift=25pt] {%
  $\frac{13\sqrt{3} +13\sqrt7 + \sqrt{143}}{26} 
    = 2.6488\dots$ \\
  (an accumulation point)
} -- (9, -0.2);
\end{tikzpicture}
\end{center}
    \caption{Gaps in $\mathscr{M}(\mathbf{H}_6)$. (This figure is not drawn to scale.)}
    \label{fig:my_label}
\end{figure}
Note from \cite{Sch76} and \cite{CCGW22} that $4/\sqrt3$ is the smallest accumulation point of $\mathscr{M}(\mathbf{H}_6)$.
Therefore, the two maximal gaps in Theorem~\ref{thm1} are in the middle part of $\mathscr{M}(\mathbf{H}_6)$.
See Figure~\ref{fig:my_label}.

As for the Hecke groups of higher index $q > 6$,
we speculate that some properties of $\mathscr{M}(\mathbf{H}_6)$
would continue to be true.
For example, Vulakh showed in \cite{Vul97} that the smallest accumulation point of $\mathscr{M}(\mathbf{H}_q)$ for even $q\ge 4$ is $2/\cos(\frac{\pi}q)$.
%$\frac2{\cos({\pi}/q)}$. 
We expect that the Hausdoff dimension past this value becomes positive for all (even) $q>6$, that is, an analogue of Theorem 1.1 holds true. 
However, the authors are unsure how maximal gaps in $\mathscr{M}(\mathbf{H}_q)$ are distributed for higher $q$ in general,
as this would require more detailed analysis based on our digit expansion.

%{\color{red}
%As investigated in \cite{Vul97}, the discrete part of the Markoff spectrum of $\mathbf H_q$ for even $q \ge 6$ shares a similar structure. 
%For odd $q \ge 5$, the first accumulation point is  
%We expect that analogues to Theorem 1.1 hold for even $q \ge 8$.
%However, finding gaps requires more detailed and careful analysis.
%}

In Section~\ref{Sec:Romik}, 
%introduce the $\mathbf{H}_6$-expansion of a point in $\mathbb R \cup \{ \infty \}$.
%Using this, 
we identify a geodesic in $\mathbb{H}$ with a bi-infinite $\mathbf{H}_6$-sequence and, 
from this, we show that Markoff and Lagrange numbers can be computed by using bi-infinite $\mathbf{H}_6$-sequences. 
%Proofs of Theorem~\ref{thm_dimension} and Theorem~\ref{thm1} are given in Section~\ref{Sec:dim} and Section~\ref{Sec:Gaps}.
%{\color{blue}INSTEAD OF THE PREVIOUS SENTENCE, 
We prove Theorem~\ref{thm1} in Section~\ref{Sec:Gaps}. 
Finally, Theorem~\ref{thm_dimension} is proven in Section~\ref{Sec:dim}.

\section{Expansion of real numbers by the Hecke group}\label{Sec:Romik}

The goal of this section is to define a digit expansion, which we call the \emph{$\mathbf H_6$-expansion}, for every oriented geodesic in the upper half plane $\mathbb H$. 
In \cite{Ros54}, Rosen developed a continued fraction algorithm for the Hecke group.
See \cite{Nak95}, \cite{SS95}, \cite{MM10} and \cite{Pan22} for other developments.
In this article, we follow the approach given in Hass and Series \cite{HS86} and Series \cite{Ser88}.
In particular, see \S2.3 in \cite{HS86}.

Let $\mathbf{H}_6$ be the subgroup of $\mathrm{PSL}_2(\mathbb{R})$ generated by  
$$
S = \begin{pmatrix}
0 & -1  \\ 1 & 0
\end{pmatrix} \quad \text{and } \quad
T = \begin{pmatrix}
1 & \sqrt 3 \\ 0 & 1
\end{pmatrix}.
$$
Indeed, we have
\begin{multline*}
\mathbf{H}_6 = \left\{ \begin{pmatrix} a & \sqrt 3 b \\ \sqrt 3 c & d\end{pmatrix} \, \big| \ ad-3bc = 1, \, a, b,c,d \in \mathbb Z \right\}  \\
\cup \left\{ \begin{pmatrix} \sqrt 3 a & b \\ c & \sqrt 3 d\end{pmatrix} \, \big| \ 3ad-bc = 1, \, a,b,c,d \in \mathbb Z \right\}
\end{multline*}
and thus $\mathbb Q(\mathbf H_6) = \sqrt 3 \mathbb Q$.
See, for example, \cite{Par77}.
We let 
$\mathrm{PSL}_2(\mathbb R)$ 
act on $\mathbb H$ by the fractional linear map, i.e., 
for 
$M = \begin{psmallmatrix} a & b \\ c & d \end{psmallmatrix} \in \mathrm{PSL}_2(\mathbb R)$,
we have
$$M \cdot z = %\begin{cases} 
\dfrac{az+b}{cz+d}. %& \text{ if } \det(M) = 1, %\\[1.5ex]
%\dfrac{a\bar z+b}{c\bar z+d} & \text{ if } \det(M) = -1. \end{cases}
$$
Consider the fundament domain $\Xi$ bordered %surrounded 
by $x=0$, $x= \sqrt 3$, $|z| = 1$ and $|z - \sqrt 3 | = 1$. 
See Figure~\ref{Fig_Fund_Domain} (left) for the fundamental domain $\Xi$ of $\mathbf{H}_6$.
Let 
$$
R =  ST^{-1}= \begin{pmatrix}
0 & - 1  \\ 1 &  -\sqrt 3
\end{pmatrix}.
$$
Then 
%Note that $S^2 = I$ and 
$R^6 = I$ (as an element of $\mathrm{PSL}_2(\mathbb R)$).
Let $\bm\delta_0$ be the geodesic in $\mathbb{H}$ connecting $0$ and $\infty$
%in $\hat{\mathbb R} = \partial \mathbb H$ 
and let $\bm\delta_i$ = $R^i(\bm\delta_0)$ for $i = 1,\dots, 5$.
Let $E_i = R^i S R^{-i}$ for $i = 0,1,\dots 5$.
Then $E_i$ is the involution that preserves $\bm\delta_i$ and exchanges the two regions separated by $\bm\delta_i$.
Let $\bm\Gamma_6$ be the subgroup of $\mathbf H_6$ generated by $E_0$, $E_1$, \dots , $E_5$.
Then $\Sigma = \Xi \cup R(\Xi) \cup R^2(\Xi) \cup \dots \cup R^5(\Xi)$ is a fundamental domain of $\bm\Gamma_6$.  
See Figure~\ref{Fig_Fund_Domain} (right) for the fundamental domain $\Sigma$ of $\bm\Gamma_6$ and the geodesics $\bm\delta_i$'s.

\begin{figure}
\centering
\begin{tikzpicture}[scale=2.8]
\node[below] at (0, 0) {$\frac 01$};
\node[below] at ({sqrt(3)}, 0) {$\frac{\sqrt 3}{1}$};
\node[below] at ({sqrt(3)/2}, 0) {$\frac{\sqrt 3}{2}$};

\node[left] at ({sqrt(3)/2}, 1.2) {$\Xi$};

\draw[thick] (-.1,0) -- ({sqrt(3)+.1},0);

\draw[thick] ({sqrt(3)},1) arc (90:150:1);
\draw[thick] (0,1) arc (90:30:1);
%\draw[thick,blue] ({(sqrt(3)-1)/2},0) arc (180:0:{1/2});

\draw[dashed] ({sqrt(3)/2},1/2) -- ({sqrt(3)/2},1.5);
%\draw[thick] (0,0) arc (180:0:{1/sqrt(3)});
%\draw[thick] ({1/sqrt(3)},0) arc (180:0:{1/sqrt(3)});
\draw[thick] (0,1) -- (0,1.5);
\draw[thick] ({sqrt(3)},1) -- ({sqrt(3)},1.5);
%\draw[thick,red] (0,0) arc (180:0:{1/(2*sqrt(3))});
%\draw[thick,red] ({1/sqrt(3)},0) arc (180:0:{1/(4*sqrt(3))});
%\draw[thick,red] ({2/sqrt(3)},0) arc (0:180:{1/(4*sqrt(3))});
%\draw[thick,red] ({2/sqrt(3)},0) arc (180:0:{1/(2*sqrt(3))});
\end{tikzpicture}
%\qquad
\begin{tikzpicture}[scale=2.8]
%\node[below] at ({-sqrt(3)/2}, 0) {$-\frac{\sqrt 3}{2}$};
\node[below] at (0, 0) {$\frac 01$};
\node[below] at ({1/sqrt(3)}, 0) {$\frac{1}{\sqrt 3}$};
\node[below] at ({2/sqrt(3)}, 0) {$\frac{2}{\sqrt 3}$};
\node[below] at ({sqrt(3)}, 0) {$\frac{\sqrt 3}{1}$};
\node[below] at ({sqrt(3)/2}, 0) {$\frac{\sqrt 3}{2}$};
%\node at ({sqrt(3)/2}, 1.6) {$T$};
%\node[blue] at (-.08, 1) {$S$};
\node[left] at ({sqrt(3)/2}, 1.2) {$\Sigma$};

\node[left] at (0, .9) {$\bm\delta_0$};
\node[below] at ({1/(2*sqrt(3))}, .26) {$\bm\delta_1$};
\node[below] at ({5/(4*sqrt(3))}, .16) {$\bm\delta_2$};
\node[below] at ({7/(4*sqrt(3))}, .16) {$\bm\delta_3$};
\node[below] at ({5/(2*sqrt(3))}, .26) {$\bm\delta_4$};
\node[right] at ({sqrt(3)}, .9) {$\bm\delta_5$};

\draw[dashed] ({sqrt(3)/2},0) -- ({sqrt(3)/2},1.5);
\draw[dashed] (0,0) arc (180:0:{1/sqrt(3)});
\draw[dashed] ({sqrt(3)},0) arc (0:180:{1/sqrt(3)});

\draw[thick] (-.1,0) -- ({sqrt(3)+.1},0);

\draw[thick,blue] ({sqrt(3)},1) arc (90:172:1);
\draw[thick,blue] (0,1) arc (90:8:1);
\draw[thick,blue] ({sqrt(3)/2},1/2) arc (90:30:1/2);
\draw[thick,blue] ({sqrt(3)/2},1/2) arc (90:150:1/2);
%\draw[thick] ({sqrt(3)/2},0) -- ({sqrt(3)/2},1.5);
%\draw[thick] (0,0) arc (180:0:{1/sqrt(3)});
%\draw[thick] ({1/sqrt(3)},0) arc (180:0:{1/sqrt(3)});
\draw[thick] (0,0) -- (0,1.5);
\draw[thick] ({sqrt(3)},0) -- ({sqrt(3)},1.5);
\draw[thick] (0,0) arc (180:0:{1/(2*sqrt(3))});
\draw[thick] ({1/sqrt(3)},0) arc (180:0:{1/(4*sqrt(3))});
\draw[thick] ({2/sqrt(3)},0) arc (0:180:{1/(4*sqrt(3))});
\draw[thick] ({2/sqrt(3)},0) arc (180:0:{1/(2*sqrt(3))});
\end{tikzpicture}
\caption{The fundamental domain $\Xi$ of the group $\mathbf H_6$ on the upper half plane (left) and the fundamental domain $\Sigma$ of the group $\bm\Gamma_6$ (right).}\label{Fig_Fund_Domain}
\end{figure}
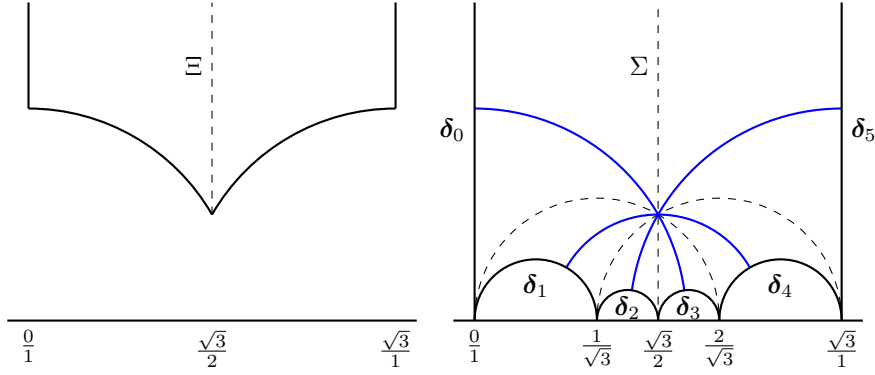

Let $\bm\gamma$ be an oriented geodesic with endpoints $\bm\gamma^+$, $\bm\gamma^-$ in $\hat{\mathbb R} := \mathbb R \cup \{ \infty \} = \partial \mathbb H$.
We will construct a symbolic sequence of $\bm\gamma$, which shows how
$\bm\gamma$ crosses %$\mathscr T$ .
%We will consider the cutting sequence of $\Sigma$ along the geodesic coding of $\bm\gamma$.
 $\mathscr T := \cup_{M \in \bm\Gamma_6} M(\partial \Sigma) = \cup_{M \in \mathbf H_6} M(\bm\delta_0)$.
Suppose that $\bm\gamma^+$, $\bm\gamma^-$ are not elements of $\mathbb Q(\mathbf H_6)$.
Then, %by intersecting $\mathscr T$ with $\bm\gamma$, 
we can divide $\bm\gamma$ into geodesic segments $ \dots, \bm\gamma_{-1}, \bm\gamma_0, \bm\gamma_1, \bm\gamma_2, \dots $, whose end points are elements of $\mathscr T$, 
along the orientation of $\bm\gamma$.
%Then, by intersecting $\mathscr T$ with $\bm\gamma$, we can divide $\bm\gamma$ into geodesic segments $ \dots, \bm\gamma_{-1}, \bm\gamma_0, \bm\gamma_1, \bm\gamma_2, \dots $ along the orientation of $\bm\gamma$.
Let $\bm\gamma_n^+, \bm\gamma_n^- \in \mathscr T$ be the endpoints of the geodesic segment $\bm\gamma_n$. 
Then $\bm\gamma_n^+ = \bm\gamma_{n+1}^-$ for all $n \in \mathbb Z$.
%along the orientation of $\bm\gamma$.
%
%Let $\bm\gamma_n^+, \bm\gamma_n^- \in \mathscr T$ be the endpoints of the geodesic segment $\bm\gamma_n$ along the orientation of $\bm\gamma$.
Note that 
the choice of $\bm\gamma_0$ is arbitrary.
After we fix $\bm\gamma_0$, 
there exists a unique $g_n \in \bm\Gamma_6$ 
for each $n \in \mathbb Z$, 
such that $\bm\gamma_n$ belongs to the ideal hexagon $g_n(\Sigma)$.
By choosing $M_n = g_nR^{i_n} \in \mathbf H_6$ for $\bm\gamma_n^- \in g_n (\bm\delta_{i_n}) = g_n R^{i_n}(\bm\delta_0)$, we have 
\begin{equation}\label{dn}
\bm\gamma_n^- \in M_n(\bm\delta_{0}), \qquad \bm\gamma_n^+ \in M_n(\bm\delta_{d_n}) \qquad \text{ for some } d_n \in \{1,2,3,4,5\}.
\end{equation}
Define
\begin{equation}\label{defN}
N_d = R^d S \qquad \text{ for } d \in\{ 1,2,3,4,5\}.
\end{equation}
Since $\bm\gamma_n^+ \in M_n (\bm\delta_{d_n}) = M_n R^{d_n}(\bm\delta_{0}) = M_n R^{d_n}S (\bm\delta_{0}) = M_n N_{d_n} (\bm\delta_{0})$
and $\bm\gamma_n^+ = \bm\gamma_{n+1}^- \in  M_{n+1} (\bm\delta_{0})$,
we have 
\begin{equation}\label{MN}
M_{n+1} = M_n R^{d_n}S = M_n N_{d_n}.
\end{equation}
The oriented geodesic $\bm\gamma$ intersects $M_{n+1} (\bm\delta_0)$ as it passes from $M_{n}(\Sigma)$ to $M_{n+1} (\Sigma)$.  
Since 
$M_{n}(\Sigma) = M_{n+1} N_{d_n}^{-1} (\Sigma)
= M_{n+1} S R^{-d_n} (\Sigma) = M_{n+1} ( S (\Sigma))$, 
the geodesic
$M_{n+1}^{-1} (\bm\gamma)$ intersects $\bm\delta_0$ as it passes from $S(\Sigma)$ to $\Sigma$
and
$M_{n+1}^{-1}(\bm\gamma^+) \in (0, \infty)$, $M_{n+1}^{-1}(\bm\gamma^-) \in (-\infty, 0)$.
Therefore, we have 
\begin{equation}\label{end_pts}
\bm\gamma^+ \in M_{n+1} \cdot (0,\infty), \qquad
\bm\gamma^- \in M_{n+1} \cdot (-\infty,0).
\end{equation}

For each oriented geodesic $\bm\gamma$, we associate a doubly infinite sequence $(d_n)_{n \in \mathbb Z}$ with an equivalence relation 
%For two bi-infinite sequence $(a_n)_{n\in\mathbb{Z}}, (b_n)_{n\in\mathbb{Z}}$ in $\{1,2,3,4,5\}^{\mathbb{Z}}$, we give an equivalence relation 
$(a_n)_{n\in\mathbb{Z}}\sim (b_n)_{n\in\mathbb{Z}}$ if there exists an integer $k\in\mathbb{Z}$ such that $a_{n+k}=b_n$ for all $n\in\mathbb{N}$. 
We call an equivalence class a \emph{bi-infinite $\mathbf H_6$-sequence} and an element in the equivalence class is called a \emph{section} of the bi-infinite $\mathbf H_6$-sequence.
Note that the bi-infinite $\mathbf H_6$-sequence of an oriented geodesic is invariant under the $\mathrm{PSL}_2(\mathbb R)$ action of $\mathbf H_6$.

On the other hand, let $(d_n)_{n \in \mathbb Z} \in \{ 1,2,3,4,5\}^{\mathbb Z}$ be given and choose $M_0 = I$. 
Then we have an oriented geodesic $\bm\gamma$ with end points $\bm\gamma^+ \in (0,\infty)$, $\bm\gamma^- \in (-\infty,0)$.  
We write
\begin{equation}\label{endpt_exp}
[ d_{0}, d_{1}, d_{2} \dots] := \bm\gamma^+ \in (0,\infty)\quad 
\text{and}
\quad
[ \dots, d_{-2}, d_{-1} ] := \bm\gamma^- \in (-\infty,0).
\end{equation}
Note that, if $(d_n)_{n \ge 0}$ is given, then $M_n$ is determined for all $n \ge 0$ from $M_0 = I$.
%Similarly, if $(d_n)_{n < 0}$ given, then  $M_n$ for all $n < 0$ is determined.
%Since $(M_m)^{-1}(\bm\gamma)$ is an oriented geodesic with ,
In general, we have 
\begin{equation}\label{gen_exp}
\bm\gamma^+ = M_m \cdot [ d_m, d_{m+1}, d_{m+2} \dots]
\quad 
\text{and}
\quad 
\bm\gamma^- = M_m \cdot [ \dots, d_{m-3}, d_{m-2}, d_{m-1} ].
\end{equation}
We remark that a similar digit expansion was introduced and its convergence was established in Proposition 15 of \cite{CCGW22}.
This also gives the convergence of \eqref{endpt_exp}.
%We skip the proof of unique convergence of $[d_1, d_2, \dots]$.% since it is similar to \cite{CCGW22}.

We will show that 
\begin{equation}\label{rev_exp}
[ \dots, d_{-2}, d_{-1} ] = - [ d_{-1}, d_{-2}, \dots ].
\end{equation}
Note from \eqref{defN} that 
\begin{equation}
\label{N_matrix}
\begin{gathered}
N_1 = \begin{pmatrix}
1 & 0 \\ \sqrt 3 & 1
\end{pmatrix}, \
N_2 = \begin{pmatrix}
\sqrt 3 & 1\\ 2 & \sqrt 3
\end{pmatrix}, \
N_3 = \begin{pmatrix}
2 & \sqrt 3 \\ \sqrt 3 & 2
\end{pmatrix}, \\
N_4 = \begin{pmatrix}
\sqrt 3 &2  \\ 1 & \sqrt 3 
\end{pmatrix}, \
N_5 = \begin{pmatrix}
1 & \sqrt 3 \\ 0 & 1
\end{pmatrix}.
\end{gathered}
\end{equation}
Let
$$
J := \begin{pmatrix} 0 & 1 \\ 1 & 0 \end{pmatrix}, \quad 
H := \begin{pmatrix} -1 & 0 \\ 0 & 1 \end{pmatrix}
\quad \text{ and } \quad 
d^\vee := 6 - d \quad \text{ for } \ d \in \{ 1,2,\dots,5\}. 
$$
Then, we have
\begin{equation}\label{HJ}
S = HJ = JH, \qquad N_d = J N_{d^\vee} J, \qquad
N_d^{-1}= H N_d H = S N_{d^\vee} S.
\end{equation}
%\begin{equation}\label{def_check}
%(R^d)^{-1} = R^{d^\vee}, \qquad  
%S N_{d^\vee} S = S R_{d^\vee} S^2 = S (R^d)^{-1} = N_d^{-1}.
%\end{equation}
%
%
If $M_0 = I$, then for each $k \ge 0$, by \eqref{MN}, \eqref{end_pts} and \eqref{HJ}, we have 
$$
M_{k+1} = N_{d_0} N_{d_1} \cdots N_{d_{k}}, \quad M_{-k} = \left( N_{d_{-1}} \right)^{-1} \cdots \left( N_{d_{-k}} \right)^{-1} = H N_{d_{-1}} \cdots N_{d_{-k}} H
$$
and 
\begin{equation}
\label{expansion}
\begin{cases}
\bm\gamma^+ %[d_0, d_1, d_2, \dots ] 
\in N_{d_0} N_{d_1} \cdots N_{d_{k}} \cdot (0, \infty),  \\
%\left( \frac{r_k}{s_k}, \frac{p_k}{q_k} \right)
%\quad \text{ where }
%\begin{pmatrix} p_k & r_k \\ q_k & s_k \end{pmatrix} := N_{d_0} N_{d_1} \cdots N_{d_{k}}, \\
\bm\gamma^- %[d_0, d_1, d_2, \dots ] 
\in H N_{d_{-1}} \cdots N_{d_{-k}} H \cdot (-\infty, 0) 
= - N_{d_{-1}} \cdots N_{d_{-k}} \cdot (0, \infty) 
%\left( - \frac{\tilde p_k}{\tilde q_k}, - \frac{\tilde r_k}{\tilde s_k} \right)
%\quad \text{ where }
%\begin{pmatrix} \tilde p_k & \tilde r_k \\ \tilde q_k & \tilde s_k \end{pmatrix} := N_{d_{-1}} \cdots N_{d_{-k}}.\label{expansion2}
\end{cases}
\end{equation}
%Therefore, $\bm\gamma^+$ and $\bm\gamma^-$ are determined by $(d_n)_{n \ge 0}$ and $(d_n)_{n < 0}$ respectively.
%We write
%\begin{equation}\label{endpt_exp}
%[ d_0, d_1, d_2\dots] := \bm\gamma^+ \in (0,\infty),\qquad 
%[ \dots, d_{-2}, d_{-1} ] := \bm\gamma^- \in (-\infty,0).
%\end{equation}
This proves \eqref{rev_exp}.
%Therefore, by \eqref{endpt_exp}, we have 
%\begin{equation}\label{rev_exp}
%[ \dots, d_{-2}, d_{-1} ] = - [ d_{-1}, d_{-2}, \dots ].
%\end{equation}
%{\color{red}
%We note that the unique convergence of $[d_1, d_2, \dots]$ is given in \cite{CCGW22}.
%}

Next, we note that the relations \eqref{HJ} give,
for all $k \ge 1$, that
$$ [ d_1^\vee, d_2^\vee, \dots ] 
\in J N_{d_0} N_{d_1} \cdots N_{d_{k}} J  \cdot (0, \infty) = \left( \frac{1}{N_{d_0} \cdots N_{d_{k}} \cdot \infty}, \frac{1}{N_{d_0} \cdots N_{d_{k}} \cdot 0} \right).
%\left( \frac{q_k}{p_k}, \frac{s_k}{r_k} \right)
%\quad \text{ where }
%\begin{pmatrix} p_k & r_k \\ q_k & s_k \end{pmatrix} := N_{d_0} N_{d_1} \cdots N_{d_{k}}.
$$
Therefore,
\begin{equation}\label{check_exp}
[ d_1^\vee, d_2^\vee, \dots ] = \frac{1}{[ d_{1}, d_{2}, \dots ]}.
\end{equation}

Suppose that $\bm\gamma^*$ is the geodesic of reversed orientation of $\bm\gamma$ with a section of the associated bi-infinite $\mathbf H_6$-sequence $(d_n)_{n\in \mathbb Z}$.
Let $\bm\gamma^*_n = \bm\gamma_{-n}$.
Then, by \eqref{dn}, we have 
\[
(\bm\gamma^*_n)^- = \bm\gamma_{-n}^+ \in M_{-n}(\bm\delta_{d_{-n}})= M_{-n}R^{d_{-n}} (\bm\delta_0)
\]
and
\[
(\bm\gamma^*_n)^+ = \bm\gamma_{-n}^- \in M_{-n}(\bm\delta_{0}) = M_{-n}R^{d_{-n}} (\bm\delta_{d^\vee_{-n}}).
\]
We have deduced that the reversed orientation geodesic $\bm\gamma^*$ associated to the bi-infinite $\mathbf H_6$-sequence of the section $(d_{-n}^\vee)_{n \in \mathbb Z}$.
Moreover, if $M_0 = I$, then 
\begin{equation*}
(\bm\gamma^*_1)^- = \bm\gamma_{-1}^+ = \bm\gamma_{0}^- \in \bm\delta_{0} = S (\bm\delta_0), \qquad 
\bm\gamma^*_1 \subset S(\Sigma).
\end{equation*}
By \eqref{gen_exp} and \eqref{check_exp}, we have
\[
(\bm\gamma^*)^+ 
= S \cdot [ d_{-1}^\vee, d_{-2}^\vee , \dots ] = [ \dots,  d_{-2}, d_{-1}] = \bm\gamma^-
\]
and
\[
(\bm\gamma^*)^- 
= S \cdot [ \dots, d_2^\vee, d_{1}^\vee, d_{0}^\vee] = [ d_0, d_{1}, d_{2}, \dots ]= \bm\gamma^+.
\]

\begin{definition}
%We define an equivalence relation ``$\sim$'' on the set $\{1,2,3,4,5\}^{\mathbb{Z}}$ as follows.
%We say that ``$(a_n)_{n\in\mathbb{Z}}\sim (b_n)_{n\in\mathbb{Z}}$'' if there exists an integer $k\in\mathbb{Z}$ such that $a_{n+k}=b_n$ for all $n\in\mathbb{N}$. 
%Such an equivalence class will be called a \emph{bi-infinite $\mathbf H_6$-sequence} and an element in the equivalence class is called a \emph{section} of the bi-infinite $\mathbf H_6$-sequence.
We write a section of a bi-infinite $\mathbf{H}_6$-sequence as $P^*|Q$ for where $P, Q \in \{ 1, 2, 3, 4, 5 \}^{\mathbb{N}}$. 
(Here, $P^*$ is understood as an element of $\{ 1, 2, 3, 4, 5\}^{\mathbb{Z}_{\le0}}$.)
For a section $P^{*}|Q$ of $A$, we define
\[
L(P^{*}|Q):= [Q] - [P^{*}] = [Q ] + [P].
\]
\label{def:L_from_bi-sequence}
\end{definition}

\begin{definition}\label{def_M}
Define $\mathcal M(A)$ to be the maximum of two supremum values as follows:
$$ 
\mathcal M(A) = \sup_{P^{*}|Q} \max \{ L(P^{*}|Q), L({(P^{\vee})}^{*}|Q^{\vee}) \}, 
$$
where $P^{*}|Q$ runs through all sections of $A$.
If $\mathcal{M}(A) = L(P^*|Q)$, we say that $P^*|Q$ is an \emph{extremal} section of $A$.
The Markoff spectrum $\mathscr{M}(\mathbf H_6)$ is defined to be the set of all Markoff numbers $\mathcal M(A)$.
\end{definition}
%\begin{lem}[Proposition 26 in \cite{CCGW22}]
For a bi-infinite $\mathbf{H}_6$-sequence $A$, it is easy to check 
\begin{equation}\label{lem:A_conjugates}
\mathcal{M}(A) =
\mathcal{M}(A^{\vee}) =
\mathcal{M}(A^*) =
\mathcal{M}((A^*)^{\vee}).
\end{equation}
%\end{lem}

\begin{definition}\label{def_L}
We define $\mathcal L(A)$ as the maximum of two limit superior values as follows:
$$ 
\mathcal L(A) : = \limsup_{P^{*}|Q} \max \{  L(P^{*}|Q), L({(P^{\vee})}^{*}|Q^{\vee}) \}
$$
where $P^{*}|Q$ runs over all sections of $A$.
%For an infinite Romik sequence $P$, we define
%$$ 
%\mathcal L (P) := \mathcal L ( {}^\infty3 P).
%$$
The Lagrange spectrum $\mathscr{L}(\mathbf H_6)$ is defined to be the set of all Lagrange numbers $\mathcal L(A)$.
\end{definition}

%Let $P = (d_i)_{i \in \mathbb N}$ be an infinite sequence in $\{ 1,2,3,4,5\}^\mathbb N$. 
%Then define 
%$$
%\xi_P = [d_1, d_2, \dots ]
%$$

\begin{prop}\label{prop}
Let $\xi, \eta \in \hat{\mathbb R} \setminus \mathbb Q(\mathbf H_6)$ be distinct
and let $A$ be the bi-infinite $\mathbf H_6$-sequence associated to the oriented geodesic $\bm\gamma$ with the endpoints $\bm\gamma^+ := \xi$,  $\bm\gamma^- := \eta$.
Then we have 
$$
\sup_{g\in \mathbf H_6} \left| g(\xi) - g(\eta) \right| = 
\sup_{P^*|Q} \max \{ L(P^*|Q) , L((P^\vee)^*|Q^\vee)  \}
$$
where the supremum is taken over all sections $P^*|Q$ of $A$.
\end{prop}

\begin{proof}
%Let $A$ be a bi-infinite $\mathbf H_6$-sequence associated to the oriented geodesic $\bm\gamma$ with the endpoints $\bm\gamma^+ = \xi$,  $\bm\gamma^- = \eta$.
Let $P^*|Q$ be a section of $A$ with $Q = (d_m, d_{m+1}, \dots )$ and $P^* = ( \dots, d_{m-2}, d_{m-1})$.
Then by \eqref{gen_exp} and \eqref{check_exp}, we have
\[
M_m^{-1} \cdot \xi - M_m^{-1} \cdot \eta = M_m^{-1} \cdot \bm\gamma^{+} - M_m^{-1} \cdot \bm\gamma^- = [Q] - ( - [P]) = [Q] + [P]
\]
and
\[
SM_m^{-1} \cdot \xi - SM_m^{-1} \cdot \eta = SM_m^{-1} \cdot \bm\gamma^{+} - SM_m^{-1} \cdot \bm\gamma^- = S\cdot [Q] - S \cdot (-[P]) = -[Q^\vee] - [P^\vee].
\]
%$$
%\xi - \eta = \bm\gamma^{+} - \bm\gamma^- = M_m \cdot [Q] - M_m \cdot ( - [P])
%$$
%Since the bi-infinite $\mathbf H_6$-sequence $(A^*)^\vee$ is associated the reversed orientation geodesic $\tilde{\bm\gamma}$, we have
%$$\tilde{M}_{k+1} = S N_{d_{-1}^\vee} \dots N_{d_{-k}^\vee} = H N_{d_{-1}} \dots N_{d_{-k}} H S = M_{-k} S.$$
%$$\tilde{M}_{-k} = S N_{d_{0}^\vee}^{-1} N_{d_{1}^\vee}^{-1} \dots N_{d_{k}^\vee}^{-1} = N_{d_{0}} \dots N_{d_{k}} S = M_{k+1} S.$$
%$$
%(M_{1-m}S)^{-1} \cdot \xi - (M_{1-m}S)^{-1} \cdot \eta = (M_{1-m}S)^{-1} \cdot \bm\gamma^{+} - (M_{1-m}S)^{-1} \cdot \bm\gamma^- = [Q^\vee] - ( - [P^\vee]) = [Q^\vee] + [P^\vee].
%$$
Therefore,
$$
\sup_{g\in \mathbf H_6} \left| g(\xi) - g(\eta) \right| \ge
\sup_{P^*|Q} \max \{ L(P^*|Q) , L((P^\vee)^*|Q^\vee)  \}.
$$
Let $[P], [Q] >0$.
Since  
$([P]+[Q])([P]^{-1} + [Q]^{-1}) = 2 + \frac{[P]}{[Q]} + \frac{[Q]}{[P]} \ge 4$,
we have 
$$
L(P^*|Q) = [P] + [Q] \ge 2 \text{ and } L({(P^{\vee})}^{*}|Q^{\vee}) = [P^\vee] + [Q^\vee] = [P]^{-1} + [Q]^{-1} \ge 2.
$$
If $\left| g(\xi) - g(\eta) \right| \ge \sqrt 2$,
then $g(\eta) < m\sqrt 2 < g(\xi)$ or $g(\xi) < m\sqrt 2 < g(\eta)$ for some $m$.
Then, we have 
$$
T^{-m} g(\eta) < 0 < T^{-m}g(\xi) \quad \text{ or } \quad T^{-m}g(\xi) < 0 < T^{-m}g(\eta).
$$
Let $\tilde{\bm\gamma} = T^{-m}g(\bm\gamma)$.
If $\tilde{\bm\gamma}^+ > 0 > \tilde{\bm\gamma}^-$, then $|g(\xi)-g(\eta)|= 
\tilde{\bm\gamma}^+ - \tilde{\bm\gamma}^-
= L(P^*|Q)$ for a section $P^*|Q$ of $A$.
If $\tilde{\bm\gamma}^+ < 0 < \tilde{\bm\gamma}^-$, then $|g(\xi)-g(\eta)|= 
\tilde{\bm\gamma}^- - \tilde{\bm\gamma}^+
= L((P^\vee)^*|Q^\vee)$ for a section $P^*|Q$ of $A$.
Hence, we have
\[
\sup_{g\in \mathbf H_6} \left| g(\xi) - g(\eta) \right| =
\sup_{P^*|Q} \max \{ L(P^*|Q), L((P^\vee)^*|Q^\vee) \}.
\]
\end{proof}
Let $A$ be a bi-infinite $\mathbf H_6$-sequence with a section $P^*|Q$. 
We associate an oriented geodesic $\bm\gamma_{P^*|Q}$ with end points $[P^*]$ and $[Q]$.
The corresponding indefinite quadratic form $f_{P^*|Q}$ is given by
$$
f_{P^*|Q} = (x + [P]y)(x-[Q]y).
$$
\begin{prop}
\label{prop:M_and_L_from_sequences}
For a bi-infinite $\mathbf H_6$-sequence $A$ with a section $P^*|Q$, we have 
$$
\mathcal M(A) = m_{\mathbf H_6}(f_{P^*|Q}) 
\quad  \text{ and } \quad 
\mathcal L (A) = %\max \{ L_{\mathbf H_6}([P]), 
L_{\mathbf H_6}([Q]).
$$
\end{prop}

\begin{proof}
Let $f(x,y) = ax^2 + bxy + cy^2 = a(x-\xi y)(x-\eta y)$.
Then $\Delta(f) = a^2 (\xi - \eta)^2$.
For 
$$
g = \begin{pmatrix} x & z \\ y & w \end{pmatrix} \in \mathbf H_6, \qquad 
\begin{pmatrix} \alpha \\ \beta \end{pmatrix} = g^{-1} \begin{pmatrix} \xi \\ 1 \end{pmatrix}, \qquad
\begin{pmatrix} \tilde\alpha \\ \tilde\beta \end{pmatrix} = g^{-1} \begin{pmatrix} \eta \\ 1 \end{pmatrix},
$$
we have 
\begin{align*}
\frac{\sqrt{\Delta(f)}}{|f(g)|}
&= \frac{ | \xi - \eta| }{|x-\xi y| \cdot |x-\eta y|}
= \left |\frac{ \begin{vmatrix} \xi & \eta \\ 1 & 1 \end {vmatrix} }{ \begin{vmatrix} x & \xi \\ y & 1 \end {vmatrix} \cdot \begin{vmatrix} x & \eta \\ y & 1 \end {vmatrix} } \right|
= \left | \frac{\begin{vmatrix} \alpha & \tilde\alpha \\ \beta & \tilde\beta \end {vmatrix}}{ \begin{vmatrix} 1 & \alpha \\ 0 & \beta \end{vmatrix} \cdot \begin{vmatrix} 1 & \tilde\alpha \\ 0 & \tilde\beta \end{vmatrix}}  \right| 
\\
&= \left | \frac{\alpha}{\beta} - \frac{\tilde\alpha}{\tilde\beta}  \right|
= \left | g^{-1} (\xi) - g^{-1} (\eta) \right|.
\end{align*}
Recall that $f(g):= f(g\begin{psmallmatrix}
    1 \\ 0
\end{psmallmatrix})$ (cf.~\eqref{def_fg}).
Therefore, we have 
%Let $f_{P^*|Q} = (x + [P]y)(x-[Q]y)$.
%Then we have $\Delta(f_{P^*|Q}) = ([P] + [Q])^2$.
\[
m_{\mathbf H_6}(f_{P^*|Q}) =
\sup_{g\in \mathbf H_6} \frac{\sqrt{\Delta(f_{P^*|Q})}}{|f_{P^*|Q}(g)|}
= \sup_{g\in \mathbf H_6} \left| g^{-1}([Q]) - g^{-1}(-[P]) \right|.
%= \sup_{g\in \mathbf H_6} \left | g ([P]) + g([Q]) \right |.
\]
Applying Definition~\ref{def_M} and Proposition~\ref{prop}, 
we obtain $\mathcal M(A) = m_{\mathbf H_6}(f_{P^*|Q})$.

For $\mathcal{L}(A)$, we notice that, when 
$g = \left(\begin{smallmatrix} x & z \\ y & w \end{smallmatrix}\right)$,
we have
\begin{align*}
y(g)^2 \cdot \left| \xi - \frac{x(g)}{y(g)} \right|
&= y(g)^2 \cdot \left| g(g^{-1}(\xi)) - \frac{x(g)}{y(g)} \right|
= y^2 \cdot \left| \frac{x g^{-1}(\xi) + z}{y g^{-1}(\xi) + w} - \frac xy \right| \\
&= \left| \frac{ xw - yz }{g^{-1}(\xi) + w/y } \right| 
= \frac{1}{\left|g^{-1}(\xi) - g^{-1}(\infty)\right|}.
\end{align*}
Therefore, we have (cf.~\eqref{def_LG})
\[
L_{\mathbf{H}_6}(\xi) = \limsup_{g \in \mathbf{H}_6} \left| g^{-1}(\xi) - g^{-1}(\infty) \right|.
\]
This gives $\mathcal L (A) = %\max \{ L_{\mathbf H_6}([P]), 
L_{\mathbf H_6}([Q]).$
%finishes the proof for $\mathcal L(A)$.
%Since $S \in \mathbf H_6$, we have   
%$$
%m_{\mathbf H_6}(f_{P^*|Q})
%= \sup_{G \in \mathbf H_6} \left | G ([Q]) - G (-[P]) \right |  \ge 2. 
%$$
%Let $[P] = [d_0, d_{-1}, d_{-2}, \dots]$ and $[Q] = [d_1, d_2, \dots ]$.
%$$
%g_k = N_{d_1} N_{d_2} \cdots N_{d_k}, %\in \mathbf G \quad  \text{ and } 
%\qquad
%\bar g_k = N_{d_1} N_{d_2} \cdots N_{d_k} S.
%$$
%
%Therefore, we have
%\begin{align*}
%L_{\mathbf{H}_6}(\xi) &= \limsup_{g \in \mathbf{H}_6} \left| g^{-1} (\xi) - g^{-1} (\infty) \right| \\
%&= \limsup_{k \to \infty} \max \left \{ \left| g_k^{-1} (\xi) - g_k^{-1} (\infty) \right|, \, \left| (\bar g_k)^{-1}( \xi) - (\bar g_k)^{-1}(\infty) \right| \right \},
%\end{align*}
%By \eqref{N_matrix} and \eqref{HJ}, we have 
%\begin{align*}
%\left| g_k^{-1} (\xi) - g_k^{-1} (\infty) \right| &= [d_{k+1}, d_{k+2} , \dots ] -
%[ d_{k}, d_{k-1}, \dots, d_1, 1, 1, 1, \dots ], \\
%\left| (\bar g_k)^{-1} (\xi) - (\bar g_k)^{-1} (\infty) \right| &= [d^\vee_{k+1}, d^\vee_{k+2} , \dots ] -
%[ d^\vee_{k}, d^\vee_{k-1}, \dots, d^\vee_1, 5, 5, 5, \dots ],
%\end{align*}
\end{proof}

%We remark that $\mathcal M(A)$ in Definition~\ref{def_M} and $\mathcal L(A)$ in Definition~\ref{def_L} coincide the definition of $L_G(\xi)$ and $\frac{\sqrt{\delta(f)}}{\inf_{M \in \mathbf G} |f(M)|}$ for $f(x,y) = (x- \eta y ) (x - \xi y)$.

%\section{Closedness of the Markoff spectrum}\label{Sec:closedness}

%Using the fact that 
%{\color{blue}This terminology, as well as the next lemma, is due to Bombieri \cite{Bom07}.
%j\begin{lem}
 %j   \label{lem:extremal}
  %j  Let $A$ be a bi-infinite sequence. Then there exists a bi-infinite sequence $B$ with $\mathcal{M}(A) = \mathcal{M}(B)$ such that $B$ admits an extremal section $P^*|Q$.
%\end{lem}
%\begin{proof}
%Let $\{ P_j^*|Q_j \}$ be a sequence of sections of $A$ such that $L(P_j^*|Q_j)\to \mathcal{M}(A)$. 
%When the set $\{ 1, 2, 3, 4, 5\}$ is endowed with discrete topology, the product space $\{ 1, 2, 3, 4, 5\}^{\mathbb{Z}}$ is compact due to Tychonoff theorem. 
%So, without loss of generality, we may assume that $L(P_j^*|Q_j) \to L(P^*|Q)$ for some $P^*|Q \in \{ 1, 2, 3, 4, 5 \}^{\mathbb{Z}}$.
%Take $B$ to be the bi-infinite sequence of which $P^*|Q$ is a section.
%\end{proof}
%}

The upshot of Proposition~\ref{prop:M_and_L_from_sequences},
together with Definitions~\ref{def:L_from_bi-sequence} and \ref{def_M},
is that we can and will compute $\mathscr{M}(\mathbf{H}_6)$ and $\mathscr{L}(\mathbf{H}_6)$ using bi-infinite sequences in $\{1, 2, 3, 4, 5\}$. 
As a consequence, we have the following theorems,
whose classical counterparts are Theorem 8 in Chapter 1 and Theorem 1 in Chapter 3 of \cite{CF89}.

%\begin{lem}\label{MarkoffTrans}
%Let $A$ be a bi-infinite Romik sequence $A$ with a section $(a_k)_{k \in \mathbb Z}$.
%If $\mathcal M(A)$ is finite, then there exists a bi-infinite Romik sequence $B$ with a section $P^* | Q$ such that $\mathcal M(A)= \mathcal M(B) = L( P^* | Q )$.
%\end{lem}

%\begin{proof}
%By considering $A$ or $A^\vee$,
%we may assume that there exists a subsequence $\{k_{n}\}_{n\ge1}$ such that $\lim\limits_{n\to\infty} L(\dots a_{k_{n}-1}|a_{k_{n}}\dots) %\lambda_{k_{n}}(A) 
%= \mathcal M(A)$.
%Let $A_n=\dots a_{k_{n}-1}|a_{k_{n}}\dots$ be a section of $A$.
%By the compactness of the space $\{ 1,2,3,4,5\}^{\mathbb Z}$, there exists a subsequence $\{A_{n_i}\}$ converging to $P^* | Q$ which is a section of a bi-infinite Romik sequence $B$.
%By the continuity of $M$, we have $L(P^* | Q) = \mathcal M(A)$.
%\end{proof}

\begin{thm}\label{MarkoffClosed}
The Markoff spectrum $\mathscr M(\mathbf H_6)$ is closed.
\end{thm}
%\begin{proof}
%{\color{blue}
%Let $\{m_j\}$ be a convergent sequence in $\mathscr{M}(\mathbf{H}_6)$.
%Thanks to Lemma~\ref{lem:extremal}, we can write $m_j = L(P_j^*|Q_j)$ for some $ P_j^*|Q_j\in \{1, 2, 3, 4, 5\}^{\mathbb{Z}}$.
%Choose a subsequential limit, say, $P^*|Q$ of $\{ P_j^*|Q_j \}$ and
%let $A$ be a bi-infinite sequence of which $P^*|Q$ is a section. 
%Then
%$L(P_j^*|Q_j) \to \mathcal{M}(A)\in \mathscr{M}(\mathbf{H}_6)$.
%}
%\end{proof}

%\begin{proof}
%Let $M \in \mathscr{M}$ be finite.
%Choose a convergent sequence $\{M_n\}_{n\ge1}$ in $\mathscr{M}(\mathbf H_6)$. 
%By Lemma \ref{MarkoffTrans}, there exist bi-infinite Romik sequences $\{A_n\}$ with a section $P_n^* | Q_n$ 
%such that $M_n = L( P_n^* | Q_n )$ for all $n \in\mathbb{N}$. 
%By the compactness of the space $\{ 1,2,3,4,5\}^{\mathbb Z}$, we have a subsequence $\{n_i\}$ such that $P_{n_i}^* | Q_{n_i}$ converges to $P^*|Q$.
%By the continuity of $L$, $M_{n_i}$ converges to $L(P^*|Q)\le \mathcal{M}(B)$ where $B$ is a bi-infinite Romik sequence with a section $P^*|Q$.
%Hence, $\lim\limits_{i\to\infty} M_{n_i}\le\mathcal{M}(B)$
%For any section $R^*|S$ of $B$, $R^*|S$ is a limit of finite shifts of $P_{n_i}^* | Q_{n_i}$.
%Thus, $L(R^*|S)\le \lim\limits_{i\to\infty} M_{n_i}$, which implies $\mathcal{M}(B)\le\lim\limits_{i\to\infty} M_{n_i}$.
%Hence, the Markoff spectrum is closed.
%\end{proof}
\begin{thm}
\label{thm:Lagrange_subset_Markoff}
The Lagrange spectrum of $\mathbf{H}_6$ is contained in the Markoff spectrum of $\mathbf{H}_6$. In other words, $\mathscr{L}(\mathbf H_6) \subset \mathscr{M}(\mathbf H_6)$.
\end{thm}
When $\{1, 2, 3, 4, 5 \}$ is equipped
with the discrete topology, 
the product space $\{ 1, 2, 3, 4, 5 \}^{\mathbb Z}$ 
becomes compact by Tychonoff's theorem.
Then it is possible to prove the two theorems above by similar arguments to those in \cite{Bom07} and \cite{KS22}. We omit the details.
%{\color{blue}
%j.j.\begin{proof}
%Let $A$ be a bi-infinite sequence and let $\{P_j^*|Q_j\}$ be a sequence of sections of $A$ such that $L(P_j^*|Q_j) \to \mathcal{L}(A)$.
%Then $\{ P_j^*|Q_j \}$ has a subsequential limit $P^*|Q$ in $\{ 1, 2, 3, 4, 5\}^{\mathbb{Z}}$.
%Take $B$ to be the bi-infinite sequence of which $P^*|Q$ is a section. 
%Then $\mathcal{L}(A) = \mathcal{M}(B)$.
%\end{proof}
%}

%\begin{proof}
%For any bi-infinite Romik sequence $A$, there exists a sequence of sections $\{P_n^* | Q_n\}$ such that $\mathcal{L}(A) = \lim\limits_{n \to \infty} L (P_n^* | Q_n)$.
%By the compactness of the space $\{ 1,2,3,4,5\}^{\mathbb Z}$, there exists a subsequence $\{n_i\}$ such that $P_{n_i}^* | Q_{n_i}$ converges to $P^* | Q$ where $B$ is a bi-infinite Romik sequence with a section $P^* | Q$.
%By the continuity of $L$, we deduce that $\mathcal{L}(A)=L(P^* | Q)\le\mathcal{M}(B)$.
%For any section $R^*|S$ of $B$, $R^*|S$ is a limit of finite shifts of $P_{n_i}^* | Q_{n_i}$.
%Thus, $L(R^*|S)\le \mathcal{L}(A)$, which concludes $\mathcal{L}(A)=\mathcal{M}(B)$.
%\end{proof}

\section{Maximal gaps of the Markoff spectrum}\label{Sec:Gaps}
%{\color{blue} LET'S JUST ERASE THIS SENTENCE. From now on, a \emph{sequence} means a sequence in the set $\{ 1, 2, 3, 4, 5\}$ unless specified otherwise.}

For any finite sequence $w$, we denote by $w^k$ the $k$ consecutive concatenation $w\cdots w$. 
Similarly, $w^{\infty} \in \{ 1, 2, 3, 4, 5 \}^{\mathbb{N}}$ will mean the infinite sequence with period $w$
and ${}^{\infty}w^{\infty} \in \{1, 2, 3, 4, 5\}^{\mathbb{Z}}$ is the bi-infinite sequence with period $w$.
For example, 
\begin{align*}
(234)^{\infty} &= 234234234\cdots, \\
{}^{\infty}(234)^{\infty} &= \cdots 234234234\cdots, \\
15(234)^{\infty} &= 15234234234\cdots.
\end{align*}
This notational convention requires some care when used for \emph{reversed} infinite sequences, namely, for elements of $\{1, 2, 3, 4, 5\}^{\mathbb{Z}_{\le0}}$.
For a finite sequence $w$, we write ${}^{\infty}w$ 
to mean the (reversed) infinite sequence $\cdots w^*w^*$, not $\cdots ww$.
For example, ${}^{\infty}(234) = \cdots 234234$.
As a result, ${}^{\infty}w | w^{\infty}$ is a section of ${}^{\infty} w^{\infty}$.

Recall that an open interval $(a, b) \subset \mathbb{R}$ is called a \emph{maximal gap} of $\mathscr{M}(\mathbf{H}_6)$ if $a, b \in \mathscr{M}(\mathbf{H}_6)$ and $(a,b)\cap\mathscr{M}(\mathbf{H}_6)$ is an empty set.

\begin{thm}\label{thm:three_gaps1}
The two intervals
\begin{align*}
   \left( \frac{\sqrt{143}}{5}, \sqrt7 \right) &= (2.391\dots, 2.645\dots),\\
   \left( \sqrt7, \frac{13\sqrt3 + 13\sqrt7 + \sqrt{143}}{26} \right) &= (2.645\dots, 2.648\dots), \\
\end{align*}
are maximal gaps of $\mathscr{M}(\mathbf{H}_6)$.
Furthermore, 
$\frac{\sqrt{143}}{5}$
and $\frac{13\sqrt3 + 13\sqrt7 + \sqrt{143}}{26}$
are accumulation points of $\mathscr{M}(\mathbf{H}_6)$.
\end{thm}
\begin{rem}
It is also possible to show that 
\[
   \left( \frac{2\sqrt{506}}{19}, \frac{2\sqrt{2803333}}{1405}\right) = (2.367\dots, 2.383\dots)
\]
is a maximal gap.
%\[
%\mathcal{M}({}^{\infty}(4343223)^{\infty}) =  \frac{2\sqrt{506}}{19}
%\quad 
%\text{and}
%\quad 
%\mathcal{M}({}^{\infty}(4343223)^{\infty}) = \frac{2\sqrt{2803333}}{1405}.
%\]
Its proof is based on a similar strategy to that of Theorem~\ref{thm:three_gaps1} but much more involved.
For brevity, we omit the proof.
\end{rem}

\begin{thm}\label{thm:longest_gap}
The interval $(\sqrt{143}/5, \sqrt7)$ is the longest maximal gap in $\mathscr{M}(\mathbf{H}_6)$.
\end{thm}
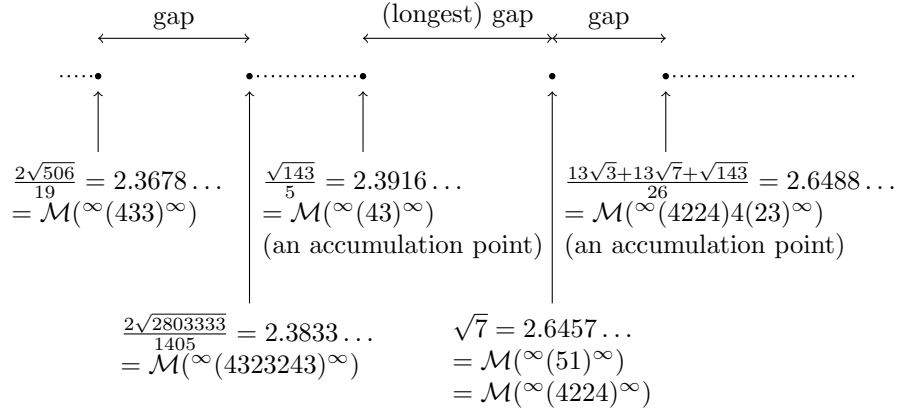
\begin{figure}
\begin{center}
\begin{tikzpicture}[every text node part/.style={align=left}]
  \draw[thick, dotted] (1, 0) -- (1.5, 0);
  \draw[thick, dotted] (3.5, 0) -- (5, 0);
  \draw[thick, dotted] (9, 0) -- (11.5, 0);
  \filldraw (1.5,0) circle (1pt);
  \filldraw (3.5,0) circle (1pt);
  \filldraw (5,0) circle (1pt);
  \filldraw (7.5,0) circle (1pt);
  \filldraw (9,0) circle (1pt);
  
  \draw[thin, <->] (1.5, 0.5) -- node[midway, above]{gap} (3.5, 0.5);
  \draw[thin, <->] (5, 0.5) -- node[midway, above]{(longest) gap} (7.5, 0.5);
  \draw[thin, <->] (7.5, 0.5) -- node[midway, above]{gap} (9, 0.5);

  \draw[thin, ->] (3.5, -3) node[below] {%
    $\frac{2\sqrt{2803333}}{1405} =2.3833\dots$ 
    \\ 
    $=\mathcal{M}({}^{\infty}(4323243)^{\infty})%
  $} -- (3.5, -0.2); 
  \draw[thin, ->] (1.5, -1) node[below, xshift = 8pt] {%
    $\frac{2\sqrt{506}}{19} 
    = 2.3678\dots$ \\
  $=\mathcal{M}({}^{\infty}(433)^{\infty})$ 
} -- (1.5, -0.2);

  \draw[thin, ->] (7.5, -3) node[below] {%
    $\sqrt{7} =2.6457\dots$ 
    \\ 
    $=\mathcal{M}({}^{\infty}(51)^{\infty}) $ \\
    $=\mathcal{M}({}^{\infty}(4224)^{\infty})$ 
} -- (7.5, -0.2); 
  \draw[thin, ->] (5, -1) node[below, xshift=15pt] {%
    $\frac{\sqrt{143}}{5} 
    = 2.3916\dots$ \\
  $=\mathcal{M}({}^{\infty}(43)^{\infty})$ \\
  (an accumulation point)
} -- (5, -0.2);

\draw[thin, ->] (9, -1) node[below, xshift=25pt] {%
  $\frac{13\sqrt{3} +13\sqrt7 + \sqrt{143}}{26} 
    = 2.6488\dots$ \\
  $=\mathcal{M}({}^{\infty}(4224)4(23)^{\infty})$  \\
  (an accumulation point)
} -- (9, -0.2);
\end{tikzpicture}
\end{center}
    \caption{Maximal gaps and their boundary points in $\mathscr{M}(\mathbf{H}_6)$. (This figure is not drawn to scale.)}
    \label{fig:H6_with_gaps}
\end{figure}
Three maximal gaps of $\mathscr{M}(\mathbf{H}_6)$ and their boundary points are shown in Figure~\ref{fig:H6_with_gaps}.

\subsection{Computational preliminaries}
Before we begin proving Theorems~\ref{thm:three_gaps1} and \ref{thm:longest_gap}, we give some computational lemmas.
%First, we cite a few lemmas from \cite{CCGW22} (without proof).
\begin{lem}%[Proposition 32 in \cite{CCGW22}]
\label{lem:w_infty_formula}
Suppose that $w = (d_1, \dots, d_k)$ is a finite sequence and let
\[
N_w:=
N_{d_1} \cdots N_{d_k} 
=
\begin{pmatrix}
a_w & b_w \\ c_w & d_w
\end{pmatrix}.
\]
Also, write $\Delta_w = \mathrm{Tr}(N_w)^2 - 4$.
Then
\[
[ w^{\infty} ] = 
\frac{a_w - d_w + \sqrt{\Delta_w}}{2c_w}.
\]
\end{lem}
\begin{proof}
Define
    \[
    x =  \frac{a_w - d_w + \sqrt{\Delta_w}}{2c_w}.
    \]
    It is enough to show that the vector $\displaystyle \begin{pmatrix} x \\ 1 \end{pmatrix}$
    is an eigenvector of $N_w$. 
    An elementary calculation shows that 
    \[
    \lambda =  \frac{a_w + d_w + \sqrt{\Delta_w}}{2}
    \]
    is an eigenvalue of $N_w$.
    Then
    \[
    N_w
    \begin{pmatrix}
        \lambda - d_w \\ c_w
    \end{pmatrix}
    =\lambda
    \begin{pmatrix}
        a_w - \lambda^{-1} \\ c_w
    \end{pmatrix}
    =\lambda
    \begin{pmatrix}
        \lambda - d_w \\ c_w
    \end{pmatrix}.
    \]
    Finally, we have $x = (\lambda - d_w)/c_w$.
    This completes the proof.
\end{proof}
Using \eqref{N_matrix} and Lemma~\ref{lem:w_infty_formula}, we can compute $[P]$ for any eventually periodic sequence $P$. This will be used tacitly throughout this section.
This also gives an easy recipe for finding $\mathcal{M}(A)$ for any periodic $A$. 
See Table~\ref{tab:Markoff_numbers_examples} for Markoff numbers of the boundary points of maximal gaps.
\begin{table}
    \centering
    \[
    \begin{array}{lll}
    \toprule
    A & \text{Extremal Section of $A$} & \mathcal{M}(A) \\
    \midrule
{}^{\infty}(43)^{\infty} &
{}^{\infty}(43)|(43)^{\infty}
& \frac{\sqrt{143}}{5} \\
{}^{\infty}(4224)^{\infty},
{}^{\infty}(51)^{\infty}
&
{}^{\infty}(4224)|(4224)^{\infty},
{}^{\infty}(51)|(51)^{\infty}
& \sqrt{7} \\
{}^{\infty}(4224)4(23)^{\infty} & 
{}^{\infty}(4224)|4(23)^{\infty} 
& 
   \frac{13\sqrt3 + 13\sqrt7 + \sqrt{143}}{26} \\
{}^{\infty}(433)^{\infty} &
{}^{\infty}(433)| (433)^{\infty} 
&
\frac{2\sqrt{506}}{19} \\
{}^{\infty}(4323243)^{\infty} & 
{}^{\infty}(4323243)|(4323243)^{\infty} 
&
\frac{2\sqrt{2803333}}{1405} \\
    \bottomrule
    \end{array}
    \]
    \caption{Some Markoff numbers $\mathcal{M}(A)$ that are boundary points of maximal gaps.}
    \label{tab:Markoff_numbers_examples}
\end{table}

\begin{lem}[cf.~Proposition 30 in \cite{CCGW22}]
\label{lem:monotone_lemma}
Suppose that $P$ and $P'$ are infinite sequences and $w$ is a finite sequence.
We have 
$[w, P] \le [w, P']$
if and only if $[P] \le [P']$.
\end{lem}
%Compare it with Proposition 30 in \cite{CCGW22}.

Next, we prove that, if $P$ is any infinite sequence \emph{containing neither 1 nor 5}, then 
\begin{equation}\label{eq:1_5_forbidden_bound}
[2^{\infty}] 
\le [P]
\le
[4^{\infty}].
\end{equation}
In other words, if we let 
$\mathcal{I} = \{2, 3, 4\}^{\mathbb{N}}$, then
\eqref{eq:1_5_forbidden_bound} holds for any $P\in\mathcal{I}$.
We will construct, say, 
$P_{\mathrm{max}}\in\mathcal{I}$ 
so that  
$[P_{\mathrm{max}}]$ 
is the sup of $\{ [P] \mid P\in\mathcal{I}\}$. 
To do this, let $P_d$ for $d = 1, \dots, 5$ be an (arbitrary) infinite sequence whose leading digit is $d$. 
Notice that %(see \S4.2 and Figure 9 in \cite{CCGW22} for example) 
\begin{equation}\label{eq:monoton_leading_digit}
[P_1] \le 
\frac1{\sqrt3} \le
[P_2] \le 
\frac{\sqrt3}2 \le
[P_3] \le 
\frac2{\sqrt3} \le
[P_4] \le 
\sqrt3 \le
[P_5].
\end{equation}
So, the leading digit of 
$P_{\mathrm{max}}$ 
must be 4, as the digit 5 is forbidden in $\mathcal{I}$.
Applying \eqref{eq:monoton_leading_digit} with Lemma~\ref{lem:monotone_lemma}, the second digit of $P_{\mathrm{max}}$ must be 4 because 5 is forbidden in $\mathcal{I}$.
Repeating this process, we obtain
$P_{\mathrm{max}} = 4^{\infty}$.
A similar argument proves that, if we let 
$P_{\mathrm{min}} = 2^{\infty}$,
$[P_{\mathrm{min}}]$
is the inf of $\{ [P] \mid P\in\mathcal{I}\}$. 
This completes the proof of \eqref{eq:1_5_forbidden_bound}.

This process can be easily modified to give a general algorithm for building
$P_{\mathrm{max}}$ and $P_{\mathrm{min}}$ with a prescribed prefix and forbidden sequences. 
We give a complete statement in the following lemma.
\begin{lem}\label{alg:min_max}
Suppose that we are given a finite sequence $w$ (which we call a \emph{prefix}) and a finite set $\{ w_1, \dots w_k\}$ of finite sequences (which we call \emph{forbidden} sequences).
Let $\mathcal{I}$ be the set of all infinite sequences beginning with $w$ and not containing any of $w_1, \dots, w_k$. 
Then either $\mathcal{I}$ is an empty set or there exist eventually periodic infinite sequences $P_{\mathrm{min}}$ and 
$P_{\mathrm{max}}$ in $\mathcal{I}$ such that
\[
[P_{\mathrm{min}}] \le [P] \le [P_{\mathrm{max}}]
\]
for any $P\in\mathcal{I}$.
\end{lem}
\begin{lem}\label{lem:contain24_or_42}
Suppose that $A$ is a bi-infinite $\mathbf{H}_6$-sequence that contains neither 1 nor 5.
Then the sequences $22$ and $44$ are forbidden in $A$ if and only if 
$\mathcal{M}(A) \le \frac{\sqrt{143}}{5} = 2.391652\dots$.
\end{lem}
\begin{proof}
For any infinite sequence $P$ containing neither 1 nor 5, 
we have $[4P] \ge [4 2^{\infty}]$ from \eqref{eq:1_5_forbidden_bound}.
Therefore, if $A$ contains 44, then
\begin{align*}
\mathcal{M}(A) &\ge L(P^*4|4Q) = [4P] + [4Q] \\
&\ge 
2\cdot [4 2^{\infty}]
=
2
\cdot
\frac{
\sqrt3 \frac1{\sqrt2} + 2
}{
\frac1{\sqrt2} + \sqrt3
}
=2.644146\dots \\
&>\frac{\sqrt{143}}{5}. 
\end{align*}
If $A$ contains $22$, we apply the same process to $A^{\vee}$ to show that
$\mathcal{M}(A) >\frac{\sqrt{143}}{5}$.
This shows that $22$ and $44$ are forbidden in $A$.

Conversely, assume that 22 and 44 are forbidden in $A$.
We will show that $L(P^*|Q)$ for any section $P^*|Q$ is bounded (from above) by 
$\frac{\sqrt{143}}{5}$.
Since $A$ doesn't contain $44$ it is impossible for both $P$ and $Q$ to have $4$ as prefix. 
So $L(P^*|Q)$ is maximized when the prefixes of $P$ and $Q$ are, say, 3 and 4, respectively. 
Since $P$ cannot contain $44$ we see that $[P]$ is bounded by $[3434... ] = [(34)^{\infty}]$. 
The same argument proves $[Q] \le [(43)^{\infty}]$.
Then
\[
[P] \le [(34)^{\infty}]
=
\frac{\sqrt{143}-\sqrt3}{10}
\text{ and }
[Q] \le [(43)^{\infty}]
=\frac{\sqrt{143}+\sqrt3}{10}.
\]
This proves that
\[
\mathcal{M}(A) \le 
\frac{\sqrt{143}-\sqrt3}{10}
+
\frac{\sqrt{143}+\sqrt3}{10}
=\frac{\sqrt{143}}5.
\]
%Therefore, if $A$ has a section of type 
%$P^*3|4Q$, 
%then the above inequalities give
%\[
%L(P^*3|4Q) = [3 P] + [4 Q] \le  
%[(34)^{\infty}]
%+
%[(43)^{\infty}]
%=\frac{\sqrt{143}}5.
%\]
%Similarly, a section of type $P^*4|3Q$ results in the same upper bound.
%
%For all other types of sections $P^*d_1|d_2Q$ of $A$, 
%it is easy to see that $\sqrt{143}/5$ is an upper bound of $L(P^*d_1|d_2Q)$. 
%For instance, 
%\[
%L(P^*3|3Q) = [3 P] + [3Q] \le 
%2
%\cdot
%\frac{\sqrt{143}-\sqrt3}{10}
%<\frac{\sqrt{143}}5.
%\]
%And
%\[
%L(P^*3|2Q) = [3 P] + [2Q] \le 
%\frac{\sqrt{143}-\sqrt3}{10}
%+
%\frac{\sqrt3}2<\frac{\sqrt{143}}5.
%\]
%Here, we used a coarser bound $[2Q]\le\sqrt{3}/2$ from \eqref{eq:monoton_leading_digit}.
%Also,
%\[
%L(P^*2|4Q) = [2 P] + [4 Q] \le 
%\frac{\sqrt3}2
%+
%\frac{\sqrt{143}+\sqrt3}{10}
%<\frac{\sqrt{143}}5.
%\]
%This shows that $\sqrt{143}/5$ is an upper bound of $\mathcal{M}(A)$, as desired.
\end{proof}
\begin{lem}\label{lem:contain1_or_5}
Let $A$ be a bi-infinite $\mathbf{H}_6$-sequence.
If $A$ contains 1 or 5, then either 
$A = {}^\infty(51)^{\infty}$ 
or 
%\[
%\mathcal{M}(A)\ge \frac{7 - \sqrt5}{\sqrt3} = 2.750\dots.
%\]
%BC: I think we can improve this as:
\[
\mathcal{M}(A)\ge 
\frac{\sqrt5 + 5}{2\sqrt3} +\frac1{\sqrt2}
=
2.79597967852851\dots.
\]
\end{lem}
\begin{proof}
Suppose that $A$ contains 5.
Assume that
\begin{equation}\label{eq:A_smaller_than2750}
\mathcal{M}(A) %< \frac{7 - \sqrt5}{\sqrt3}
< 
\frac{\sqrt5 + 5}{2\sqrt3} +\frac1{\sqrt2}
=
2.79597967852851\dots.
\end{equation}
We will deduce $A = {}^\infty(51)^{\infty}$ from this.

If the digit 5 in $A$ extends to 55, then \eqref{eq:monoton_leading_digit} gives 
\[
\mathcal{M}(A) \ge L(P^*5|5Q) = [5P] + [5Q] \ge \sqrt3 + \sqrt 3 = 2\sqrt 3.
\]
This contradicts \eqref{eq:A_smaller_than2750}.
Similarly, if $5$ extends to 54, 
\[
\mathcal{M}(A) \ge 
L(P^*5|4Q) = [5 P] + [4 Q]
\ge \sqrt3 + \frac2{\sqrt3}=
2.886751\dots >
\frac{\sqrt5 + 5}{2\sqrt3} +\frac1{\sqrt2}
.
\]
If $5$ extends (to the left) to 45, 
\[
\mathcal{M}(A) \ge L(P^*4|5Q) = [4 P] + [5Q]
\ge 
\frac2{\sqrt3}  + \sqrt3 
>
\frac{\sqrt5 + 5}{2\sqrt3} +\frac1{\sqrt2}
.
\]
Therefore, $55$, $54$ and $45$ are forbidden in $A$. 
%Applying the same argument to $A^{\vee}$, we see that $11, 12, 21$ are also forbidden in $A$.
Suppose that the digit 5 in $A$ extends to 53. 
We will find a lower bound of
$
L(P^*5|3Q) = [5 P] + [3Q]
$.
Recall that $55$, $54$ and $45$ are forbidden in $A$. 
Applying the same argument to $A^{\vee}$, we see that $11, 12, 21$ are also forbidden in $A$.
We apply Lemma~\ref{alg:min_max} with the prefix 5 and 
$\{55, 54, 45, 11, 12, 21\}$ as forbidden sequences 
and conclude that
\[
[5 P ]\ge [5(13)^{\infty}] = 
\frac{\sqrt5 + 5}{2\sqrt3}.
\]
Also, from Lemma~\ref{alg:min_max} with the prefix 3 instead, we obtain 
\[
    [3Q]\ge [(31)^{\infty}] = 
    \frac{\sqrt5 + 1}{2\sqrt3}.
\]
So,
\begin{align*}
\mathcal{M}(A) &\ge 
L(P^*5|3Q) 
\ge
[5(13)^{\infty}] + 
[(31)^{\infty}]
\\
&=
\frac{\sqrt5 + 5}{2\sqrt3}
+
    \frac{\sqrt5 + 1}{2\sqrt3}
=\sqrt3 + \frac{\sqrt{15}}3 \\
&>
\frac{\sqrt5 + 5}{2\sqrt3} +\frac1{\sqrt2},
\end{align*}
which contradicts \eqref{eq:A_smaller_than2750}.
Next, suppose that the digit 5 in $A$ extends to $52$.
This time, we apply Lemma~\ref{alg:min_max} with the prefix 2, %(and the same forbidden sequences)
%to get $[2Q]\ge [2(13)^{\infty}]$. 
%(DH: THIS IS NOT SHARP) 
%\todo[inline]{BC: you're right because $21$ is forbidden. So, is the sharp lower bound $[2^{\infty}]?$}
yielding $[2Q] \ge [2^{\infty}] = 1/\sqrt2$.
So,
\[
\mathcal{M}(A) \ge 
L(P^*5|2Q) = [5 P] + [2 Q] \ge 
[5(13)^{\infty}] + 
[2^{\infty}]
=
\frac{\sqrt5 + 5}{2\sqrt3}
+\frac1{\sqrt2}
%2.79597967852851
%\frac{7 - \sqrt5}{\sqrt3},
\]
contradicting \eqref{eq:A_smaller_than2750}.

In conclusion, the digit 5 in $A$ must extend to $51$.
In particular, $A^{\vee}$ contains 15.
We have $\mathcal{M}(A^{\vee}) = \mathcal{M}(A)$ 
(see \eqref{lem:A_conjugates}) 
and therefore we can apply the same argument to $A^{\vee}$ and conclude that $15$ in $A^{\vee}$ must extend to $151$.
Continuing this way, we see that the (original) 5 in $A$ must extend to $(51)^{\infty}$.
Again, we can apply the same argument to $A^*$, which results in
$A = {}^{\infty}(51)^{\infty}$.
This completes the proof of the lemma.
\end{proof}
\begin{lem}\label{lem:423_424}
Suppose that $A$ is a bi-infinite $\mathbf{H}_6$-sequence containing neither 1 nor 5. 
If $A$ contains $443$,
\[
\mathcal{M}(A) \ge 
\frac{
2\sqrt2 + \sqrt3}{\sqrt6 + 1
}
+
\frac{
7\sqrt2 + 4\sqrt3}{3\sqrt6 + 5
}
=2.68480\dots .
\]
If $A$ contains $444$,
\[
\mathcal{M}(A) \ge 
\frac{
2\sqrt2 + \sqrt3}{\sqrt6 + 1
}
+
\frac{
5\sqrt2 + 8\sqrt3
}{
2(\sqrt6 + 5)
}
=2.72669\dots . 
\]
\end{lem}
\begin{proof}
First, from \eqref{eq:1_5_forbidden_bound} and Lemma~\ref{lem:monotone_lemma}, 
we have
\[
[4 P] \ge [4 2^{\infty}]
=
\frac{\sqrt3 + 2 \sqrt2}{\sqrt6 + 1}
\]
and
\[
[43 Q] \ge [432^{\infty}] = 
\frac{7\sqrt 2 + 4\sqrt3}{3\sqrt6 + 5}.
\]
So, if $A$ contains 443,
\[
L(P^*4|43Q) = [4 P] + [43 Q]
\ge
\frac{\sqrt3 + 2 \sqrt2}{\sqrt6 + 1}
+
\frac{7\sqrt 2 + 4\sqrt3}{3\sqrt6 + 5}.
\]
This proves the first part of the lemma.

When $A$ contains $444$, we use
\[
[44Q] \ge [442^{\infty}]
=
\frac{ 5\sqrt2 + 8\sqrt3 }{ 2(\sqrt6 + 5) },
\]
so that
\[
L(P^*4|44Q) = [4 P] + [44 Q]
\ge
\frac{\sqrt3 + 2 \sqrt2}{\sqrt6 + 1}
+
\frac{ 5\sqrt2 + 8\sqrt3 }{ 2(\sqrt6 + 5) }.
\]
\end{proof}
%\begin{lem}\label{lem:contain24_or_42}
%Suppose that $A\in \{ 2, 3, 4\}^{\mathbb{Z}}$ is a bi-infinite Romik sequence.
%Then $A$ contains $24$ or $42$ if and only if $\mathcal{M}(A) > \frac{\sqrt{143}}{5}$.
%\end{lem}
\begin{lem}\label{lem:contain242}
Suppose that $A$ is a bi-infinite $\mathbf{H}_6$-sequence whose digits contain neither 1 nor 5.
Also, suppose that $A$ does not contain any of the following sequences
\begin{equation*}
444, 443, 344, 322, 223, 222.\tag{$\dagger$}
\end{equation*}
If $A$ contains 4423, then
\begin{align*}
\mathcal{M}(A) &\ge 
\frac{7 \, \sqrt{299}  + 69 \, \sqrt{3}}{3\sqrt{897} + 92} +
\frac{2 \, {\left(\sqrt{143}  + 6 \, \sqrt{3}\right)}}{\sqrt{429} + 13} \\
&= 2.64876844390582 \dots.
\end{align*}
If $A$ contains 4424, then
\begin{align*}
\mathcal{M}(A) &\ge 
\frac{7 \, \sqrt{299}  + 69 \, \sqrt{3}}{3\sqrt{897} + 92} +
\frac{76 \, \sqrt{299} + 759 \, \sqrt{3}}{33 \, \sqrt{897} + 989}
\\
& = 
2.65226159739540\dots.
\end{align*}
%Then $\mathcal{M}(A) \ge \sqrt7$ with the equality holding if and only if $A = {}^{\infty}(2244)^{\infty}$.
\end{lem}
\begin{proof}
Suppose that $A$ contains 4423. 
We find a lower bound of
\[
L(P^*4|423Q) = [4 P ] + [423Q]
\]
by applying Lemma~\ref{alg:min_max}.
Using the prefix $4$ and the forbidden sequences in ($\dagger$),
we obtain 
$[4 P] \ge [42(242)^{\infty}]$.
Also, using the prefix $423$, we have
$[423Q] \ge [4(23)^{\infty}]$.
As a result, we obtain
\begin{align*}
L(P^*4|423Q) &\ge [42(242)^{\infty}] + [4(23)^{\infty}] \\
&=
\frac{7 \, \sqrt{299}  + 69 \, \sqrt{3}}{3\sqrt{897} + 92} +
\frac{2 \, {\left(\sqrt{143}  + 6 \, \sqrt{3}\right)}}{\sqrt{429} + 13}.
\end{align*}
When $A$ contains 4424, we proceed similarly. 
Using the prefix $424$, we obtain
$[424Q] \ge [4242(242)^{\infty}]$,
and
\begin{align*}
L(P^*4|424Q) 
&\ge [42(242)^{\infty}] + [4242(242)^{\infty}] \\
&=
\frac{7 \, \sqrt{299}  + 69 \, \sqrt{3}}{3\sqrt{897} + 92} +
\frac{76 \, \sqrt{299} + 759 \, \sqrt{3}}{33 \, \sqrt{897} + 989}.
\end{align*}
This completes the proof of the lemma.
\end{proof}

\subsection{Proof of Theorem~\ref{thm:three_gaps1}}
First, we prove that $(\frac{\sqrt{143}}{5}, \sqrt7)$ is a maximal gap.
Suppose that $A$ be a bi-infinite $\mathbf{H}_6$-sequence and assume that 
\[
\frac{\sqrt{143}}{5}
< \mathcal{M}(A) < 
\sqrt7 = 2.64575\dots.
\]
We shall deduce a contradiction from this.

If $A$ contains 1 or 5, then Lemma~\ref{lem:contain1_or_5} implies that $A = {}^{\infty}(15)^{\infty}$, which is a contradiction because
$\mathcal{M}(A) = \sqrt{7}$.
So, we can assume that $A\in\{2, 3, 4\}^{\mathbb{Z}}$. 
Notice from Lemma~\ref{lem:423_424} that 443 and 444 are forbidden in $A$.
Since 
$\mathcal{M}(A)=
\mathcal{M}(A^{\vee})=
\mathcal{M}(A^*)=
\mathcal{M}((A^*)^{\vee})
$
(cf.~\eqref{lem:A_conjugates})
all the sequences in ($\dagger$) of Lemma~\ref{lem:contain242} are forbidden in $A$.
On the other hand, Lemma~\ref{lem:contain24_or_42} implies that $A$ must contain $22$ or $44$.
Replacing $A$ with $A^{\vee}$ if necessary, we may assume without loss of generality that $A$ contains 44.
Avoiding 443 and 444, we see that 44 must extend to 442. 
Furthermore, Lemma~\ref{lem:contain242} shows that this should further extend to 4422. 
In particular, $A^{\vee}$ contains 2244.
Repeating the same argument, we see that the (original) 44 must extend to $(4422)^{\infty}$.
to the same with $(A^*)^{\vee}$ (which contains ${}^{\infty}(2244)$)
and we conclude that $A = {}^{\infty}(4422)^{\infty}$.
This is a contradiction because one can easily show that
$\mathcal{M}({}^{\infty}(4422)^{\infty}) = \sqrt7$.
This completes the proof of the fact that $(\frac{\sqrt{143}}{5}, \sqrt7)$ is a maximal gap.

Next, we prove that $(
\sqrt7,
\frac{13\sqrt3 + 13\sqrt7 + \sqrt{143}}{26}
)$ is a maximal gap.
Assume that $A$ is a bi-infinite $\mathbf{H}_6$-sequence such that
\[
\sqrt7 < \mathcal{M}(A) < 
\frac{
13\sqrt3 + 13\sqrt7 + \sqrt{143}
}{
26
}
=
2.64883416482063
\dots.
\]
As before, Lemma~\ref{lem:contain1_or_5} says that 
$A\in\{2, 3, 4\}^{\mathbb{Z}}$.
Also, thanks to Lemmas~\ref{lem:contain24_or_42} and \ref{lem:423_424}, 
we may assume that $A$ contains $44$, but 443 and 444 are forbidden in $A$.
By the same argument, all the sequences in ($\dagger$) of Lemma~\ref{lem:contain242} are forbidden in $A$.
So, Lemma~\ref{lem:contain242} says that $4424$ is also forbidden in $A$.
In conclusion, the sequence $44$ must extend to 
$4422$ or $4423$. 

Suppose that $44$ is extended to $4423$.
Then $4423$ must extend (to the left) to 24423, because 444 and 344
are forbidden in $A$.
So,
\[
\mathcal{M}(A) \ge L(P^*24|423Q) = [42 P] + [423Q].
\]
We apply Lemma~\ref{alg:min_max} with the prefixes 42 and 423 and the sequences ($\dagger$), plus 4424, as forbidden.
This gives 
$[42 P] \ge [(4224)^{\infty}]$
and
$[423Q] \ge [4(23)^{\infty}]$.
Consequently,
\[
\mathcal{M}(A) \ge 
[(4224)^{\infty}] + 
[4(23)^{\infty}]
=
\frac{
13\sqrt3 + 13\sqrt7 + \sqrt{143}
}{
26
}
.
\]
This completes the proof that 
$(
\sqrt7,
\frac{13\sqrt3 + 13\sqrt7 + \sqrt{143}}{26}
)$ is a maximal gap.

To complete the proof of Theorem~\ref{thm:three_gaps1}, it remains to show that 
$\sqrt{143}/5$ 
and
$\frac{13\sqrt3 + 13\sqrt7 + \sqrt{143}}{26}$
are accumulation points of $\mathscr{M}(\mathbf{H}_6)$.
Since $22$ and $44$ are forbidden in ${}^{\infty}3 \, (43)^{k+l} \, 3^{\infty}$
for any $k, l\ge0$,
Lemma~\ref{lem:contain24_or_42} says 
\begin{equation}
    \mathcal{M}({}^{\infty}3 \, (43)^{k+l} \, 3^{\infty}) < 
    \frac{\sqrt{143}}{5}.
\end{equation}
On the other hand,
\begin{align*}
    \mathcal{M}({}^{\infty}3 \, (43)^{k+l} \, 3^{\infty}) 
    &\ge
    L({}^{\infty}3 \, (43)^{k} | (43)^l \, 3^{\infty})  \\
    &=
    [(34)^{k} \, 3^{\infty}] + [(43)^l\, 3^{\infty}] \\
    &\to
    [(34)^{\infty}] + [(43)^{\infty}] = 
    \mathcal{M}({}^{\infty}(43)^{\infty}) = \frac{\sqrt{143}}5
\end{align*}
as $k,l\to\infty$ (cf.~Table~\ref{tab:Markoff_numbers_examples}).
This proves that 
$\sqrt{143}/5$ is an accumulation point.

To prove that $\frac{13\sqrt3 + 13\sqrt7 + \sqrt{143}}{26}$ is an accumulation point, define
\[
A_k = {}^{\infty}(4224) \, 4(23)^k \, 3^{\infty}
\]
for $k\ge1$.
To compute $\mathcal{M}(A_k)$, it is enough to check $L(P^*|Q)$ for those sections $P^*|Q$ of $A_k$ that the prefix of $P$ and $Q$ is 4. 
From this, it follows that 
$\mathcal{M}(A_k) = L(P^*|Q)$ when $P = (4224)^{\infty}$ and $Q = 4(23)^k3^{\infty}$.
So,
\begin{align*}
\mathcal{M}(A_k) &=  [(4224)^{\infty}] + [4(23)^k 3^{\infty}]  \\
&\to
[(4224)^{\infty}] + [4(23)^{\infty}] 
=\frac{13\sqrt3 + 13\sqrt7 + \sqrt{143}}{26}
\end{align*}
as $k\to\infty$.
This proves that $\frac{13\sqrt3 + 13\sqrt7 + \sqrt{143}}{26}$ is an accumulation point.

\subsection{Proof of Theorem~\ref{thm:longest_gap}}
Recall that 
a section $P^*|Q$ of a bi-infinite $\mathbf H_6$-sequence $A$ is called extremal if $\mathcal{M}(A) = L(P^*|Q)$.
\begin{prop}\label{prop:shift_root_3}
Suppose that $P^*|Q$ is an extremal section of a bi-infinite $\mathbf H_6$-sequence $A$. 
Then, for any $k\ge0$, 
the section $P^*|5^k\, Q$ is extremal and  
\[
\mathcal{M}(P^* 5^k Q) = 
\mathcal{M}(A) +
\sqrt3 k.
\]
\end{prop}
\begin{proof}
Notice that
\begin{align*}
L(P^*| 5^k Q) &= 
[P] + [5^k Q] \\
&=[P] + [Q] + \sqrt3k \\
&= \mathcal{M}(A) + \sqrt 3k. 
\end{align*}
So, it remains to show that 
$P^*| 5^k Q$ 
is an extremal section of $P^*5^kQ$.

First, we claim that 
\[
[u \, 5 \, X] \le [5 \, u \, X]
\]
for any infinite sequence $X$ and a finite sequence $u$.
If $u = 5^l$ for some $l$, there is nothing to prove.
Otherwise, write $u = 5^l u'$ with maximal $l\ge 0$, so that the leading digit of $u'$ is not 5.
Then, from \eqref{eq:monoton_leading_digit}, we see that
\[
    [u' \, 5 \, X] \le \sqrt3 \le [5 \, u' \, X].
\]
We can apply Lemma~\ref{lem:monotone_lemma} to the above inequality to obtain
\[
[u \, 5 \, X] = [5^l \, u' \, 5 \, X] \le 
[5^l \, 5 \, u' \, X]
= [5 \, u \, X],
\]
which proves the claim.

Now we prove the extremality of 
$P^*| 5^k Q$.
It is easy to see that 
$L(P^* 5^l| 5^{k-l} Q)=L(P^*|Q) + \sqrt3 k$
for $l = 0, \dots, k$.
So it suffices to consider the two kinds of sections
$
L(P^* 5^k w | Q') 
$
and
$
L(P'^*|u^* 5^k Q) 
$
where $P = u P'$ and $Q = w Q'$ for arbitrary finite sequences $u$ and $w$.
For the first, we have
\begin{align*}
L(P^*  5^k w | Q') &= 
[w^* 5^k P] + [Q] \\
&\le 
[5^k w^* P] + [Q] 
\end{align*}
where the last inequality comes from the above claim.
Therefore,
\begin{align*}
L(P^*  5^k w | Q') &\le 
[5^k w^* P] + [Q']  \\
&= 
[w^* P] + [Q'] 
+
\sqrt3 k \\
&=
L(P^* w | Q') + \sqrt3 k \\
&\le 
L(P^*|Q) 
+
\sqrt3 k
=
\mathcal{M}(A) + \sqrt3k.
\end{align*}
Here the last inequality comes from the extremallity of $P^*|Q$.
For the other kind of sections, we proceed similarly.
\begin{align*}
L(P'^*|u^* 5^k Q) &=
[P'] + [u^* 5^k Q]  \\
&\le 
[P'] + [5^k u^* Q]  \\
&= 
[P'] + [u^* Q]  + \sqrt3k \\
&\le L(P^*|Q) + \sqrt3k = 
\mathcal{M}(A) + \sqrt3k.
\end{align*}
This completes the proof of the extremality of $P^*| 5^k Q$.
\end{proof}
\begin{table}
    \centering
    \begin{tabular}{c c l}
    \toprule
    $A_j$ & Extremal Section of $A_j$ & $\mathcal{M}(A_j)$ \\
    \midrule
    $A_1 = {}^{\infty}(42)^{\infty} $  &
    ${}^{\infty}(42) | (42)^{\infty} $ 
    & $\frac{\sqrt{13}}{\sqrt3} = 2.081\dots$ \\
    $A_2 = {}^{\infty}(42)\, 3\, (42)^{\infty} $  &
    ${}^{\infty}(42) \, 3 | (42)^{\infty} $ 
    & $\frac4{\sqrt{3}} = 2.309\dots$ \\
    $A_3 = {}^{\infty}(43)^{\infty} $  &
    ${}^{\infty}(43)| (43)^{\infty} $ 
    & $\frac{\sqrt{143}}{5} = 2.391 \dots$ \\
    $A_4 = {}^{\infty}(4224)^{\infty} $  &
    ${}^{\infty}(4224)| (4224)^{\infty} $ 
    & $\sqrt{7}= 2.645\dots$ \\
    $A_5 = {}^{\infty} 4^{\infty} $  &
    ${}^{\infty}4| 4^{\infty} $ 
    & $\sqrt{8} = 2.828\dots$ \\
    $A_6 = {}^{\infty}(51522)^{\infty} $  &
    ${}^{\infty}(51522)| (51522)^{\infty} $ 
    & $\frac{4\sqrt{26}}{7} = 2.913\dots$ \\
    $A_7 = {}^{\infty}(522)^{\infty} $  &
    ${}^{\infty}(522)| (522)^{\infty} $ 
    & $\sqrt{10} = 3.162\dots$ \\
    $A_8 = {}^{\infty}(5212)^{\infty} $  &
    ${}^{\infty}(5212)| (5212)^{\infty} $ 
    & $\sqrt{11} = 3.316\dots$ \\
    $A_9 = {}^{\infty}(534532)^{\infty} $  &
    ${}^{\infty}(534532)| (534532)^{\infty} $ 
    & $\frac{\sqrt{435}}{6} = 3.476\dots$ \\
    $A_{10} = {}^{\infty}(5254)^{\infty} $  &
    ${}^{\infty}(5254)| (5254)^{\infty} $ 
    & $\frac{2\sqrt{10}}{\sqrt3} = 3.651\dots$ \\
    \bottomrule
    \end{tabular}
    \caption{Extremal sections of $A_j$ for $j = 1, \dots, 10$ and their Markoff numbers.}
    \label{tab:10_markoff_numbers}
\end{table}
To finish the proof of Theorem~\ref{thm:longest_gap}, we find extremal sections of $A_j$ for $j =1, \dots, 10$ as shown in 
Table~\ref{tab:10_markoff_numbers},
and compute $\mathcal{M}(A_j)$ for $j = 1, \dots, 10$ using Lemma~\ref{lem:w_infty_formula}.
We see that all the differences $\mathcal{M}(A_{j+1})-\mathcal{M}(A_j)$, as well as $\mathcal{M}(A_1) + \sqrt3 - \mathcal{M}(A_{10})$, are less than or equal to $\sqrt7 - \frac{\sqrt{143}}5$. 
By combining this with Proposition~\ref{prop:shift_root_3},
we complete the proof of Theorem~\ref{thm:longest_gap}.

\section{Dimension of the Markoff spectrum}
\label{Sec:dim}
The goal of this section is to prove Theorem~\ref{thm_dimension}.
We recall the following characterization of the initial discrete part of $\mathscr{M}(\mathbf{H}_6)$. 
Note that the digit expansion used in \cite{CCGW22} is slightly different from ours.

\begin{thm}[\cite{CCGW22}]\label{disc_spec}
If $\mathcal M(A) < \frac{4}{\sqrt 3}$, then
$A$ is one of the following bi-infinite $\mathbf H_6$-sequences: 
%${}^\infty(42)3 | (42)^\infty$
$$
{}^\infty 3 ^{\infty}, \ {}^\infty(42)^{\infty}, \
\text{ or }
{}^\infty((42)^k3)^{\infty} \ \text{ for some } k \ge 1.
$$
Also, we have
$$ \lim_{k\to\infty} \mathcal M({}^\infty ((42)^k3)^{\infty}) 
= \lim_{k\to\infty} L ({}^\infty ((42)^k3) | ((42)^k3)^{\infty}) 
= \frac{4}{\sqrt 3}.$$
\end{thm}

We construct bi-infinite sequences whose Markoff numbers are %sufficiently
close to the first accumulation point $\frac{4}{\sqrt3}$
%, \frac{4}{\sqrt3} + \varepsilon$. 
by concatenating the blocks $(42)^{m+1}3$ and $(42)^{m}3$ for large $m$.
Let $\varepsilon>0$.
Choose $m\in\mathbb{N}$ such that for any $P$ and $Q$
\begin{equation}\label{ebound}
   [3(24)^{m} P] + [(42)^{m+1} Q]
   < [3(24)^{\infty}]+[(42)^{\infty}] +\varepsilon = \frac{4}{\sqrt{3}}+\varepsilon. 
\end{equation}
Let 
$$
u = (42)^{m+1}3 \quad \text{and} \quad w = (42)^{m}3,
$$
and define
\begin{equation*}
    E = \{ P \in \{ 2,3,4\}^{\mathbb N} \, | \, P = w^{m_1} u^{n_1} w^{m_2} u^{n_2} \cdots , n_i, m_i \in \{1,2\}  \text{ for } i \ge 1  \}.
%    \tilde{E} &= \{ P \in \{ 2,3\}^{\mathbb N} \, | \, P = w^{m_1} u^{n_1} w^{m_2} u^{n_2} \cdots , n_i, m_i \in \mathbb{N} \text{ for } i \ge 1 \}.
\end{equation*}
\begin{lem}\label{dimension}
We have 
$$\dim_H (\{ [P] \, | \, P \in E \}) > 0.$$
\end{lem}

\begin{proof}
Let 
$$
F := \{ [P] \, | \, P \in E \}.
$$ 
We will construct an iterated function system of $F$,
which will give a lower bound of the Hausdorff dimension of $F$. 
See \cite{Fal03}*{Proposition 9.7}, for example.

Define
$$
\alpha := [(w^2 u)^\infty ] \quad
\text{and}
\quad
\beta := [(w u^2)^\infty]. 
$$
Then, for each $P \in E$, we have
$$
%\frac{2}{\sqrt 3} = [4 1^\infty] 
0 <  \alpha \le [P] \le \beta.
$$
For $i = 1, 2, 3, 4$, define $f_i : F \to F$ to be
\begin{align*}
f_1([P]) &= [w^2\,u\,P],  &   f_2([P]) &= [w^2\,u^2\,P],  \\
f_3([P]) &= [w\,u\,P], & f_4([P]) &= [w\,u^2\,P]. 
\end{align*}
Then
$F$ is a disjoint union of the images of $f_i$, that is,
$$F = f_1(F) \cup f_2(F) \cup f_3(F) \cup f_4(F).$$
Let $a_i, b_i, c_i, d_i \in \mathbb{R}$ such that
$f_i (x)=\begin{pmatrix}
a_i & b_i \\ c_i & d_i
\end{pmatrix}\cdot x$
for $i = 1, 2, 3, 4$.
From \eqref{N_matrix}, we have $c_i, d_i > 1$ 
and 
\[
|f_i (x)-f_i (y)|=\frac{|x-y|}{(c_i x+d_i )(c_i y+d_i )} < \frac{|x-y|}{(\alpha+1)^2}.
\]
Thus, $\{f_1, f_2, f_3, f_4\}$ is a family of contracting functions on $F$.
Furthermore, for all $x,y \in F$,
$$
|f_i (x)-f_i (y)|=\frac{|x-y|}{(c_i x+d_i )(c_i y+d_i )} >  \frac{|x-y|}{(c_i \beta + d_i)^2} =C_i |x-y|
$$
where $C_i:=1/(c_i\beta + d_i)^2$.
Let $s>0$ be a constant satisfying
\begin{equation*}
C_1^s + C_2^s + C_3^s + C_4^s = 1.
\end{equation*}
Now, we invoke \cite{Fal03}*{Proposition 9.7} to conclude
$\dim_H(F) \ge s>0$.
\end{proof}

Choose 
$$P = w^{m_1} u^{n_1} w^{m_2} u^{n_2} \cdots \in E,$$
where $n_i, m_i \in \{1,2\}$.
Let
$$
v_k = w^{m_1} u^{n_1} w^{m_2} \cdots u^{n_k}
$$
and 
$$
A_P = {}^\infty w u^3 w v_1 w^{2} u^3 w v_{2} w^{3} u^3 w v_{3} w^{4} u^3 w v_4 \cdots w^{k} u^3 w v_{k} w^{k+1} u^3 w v_{k+1} \cdots.
$$

\begin{lem}\label{sections}
We have
$$
\mathcal L( A_P) = [(w^*)^\infty] + [u^3 w P].
$$
\end{lem}

\begin{proof}
To compute $\mathcal{L}(A_P)$, we claim that it is enough to consider a section of $A_P$ of the type 
$R^*3 |4 S$ of $A_P$ or
$R^*4|3S$ of $(A_P)^\vee$.
Other sections of $A_P$ are of the type $R_1^*2|S_1$ or $R_2^*4|S_2$.
For sections of the first type, we have either
\[
L(R_1^* 2 | S_1) = L(R_1^* 2 | 3S'_1) = [2 R_1 ] + [3 S'_1 ] \le \frac{\sqrt 3}{2} + \frac{2}{\sqrt 3} = \frac{7}{2\sqrt 3} < \frac{4}{\sqrt 3}
%\leq \mathcal{L}(A_P), 
\]
or
\[
L(R^*_1 2 | S_1) = L(R_1^* 2 | 42S''_1) =  [2 R_1] + [42 S''_2] \le \frac{\sqrt 3}{2} + \frac{7}{3\sqrt 3} = \frac{23}{6\sqrt 3} <\frac{4}{\sqrt 3}.
\]
For the second type, we have either
\[
L(R^*_2 4 | S_2) = L((R'_2)^* 3 4 | 2 S'_2) = [4 3 R'_2] + [2 S'_2] 
\le \frac{4\sqrt 3}{5} + \frac{\sqrt 3}{2} = \frac{13\sqrt 3}{10} < \frac{4}{\sqrt 3}
\]
or
\[
L(R^*_2 4 | S_2) = L((R''_2)^* 2 4 | 2 S'_2) = [4 2 R''_2] + [2 S'_2] 
\le \frac{7}{3\sqrt 3} + \frac{\sqrt 3}{2} = \frac{23}{6\sqrt 3} < \frac{4}{\sqrt 3}.
\]
Since $A_P$ is not one of the bi-infinite sequences listed in Theorem~\ref{disc_spec} we have $\mathcal{L}(A_P) \ge \frac4{\sqrt3}$.
%because $A_P$ is not one of the bi-infinite sequences listed in Theorem~\ref{disc_spec}.
%where %$R_2 = 2R'_2$ and 
%$S_3 = 3S'_3$.
%Here, we used  
%$$
%N_2 = \begin{pmatrix} \sqrt 3 & 1 \\ 2 & \sqrt 3 \end{pmatrix},
%N_3 = \begin{pmatrix} 2 & \sqrt 3 \\ \sqrt 3 & 2 \end{pmatrix},
%N_{42} = \begin{pmatrix} 7 & 3\sqrt 3 \\ 3\sqrt 3 & 4 \end{pmatrix}, \
%N_{43} = \begin{pmatrix} 4\sqrt 3 & 7 \\ 5 & 3\sqrt 3 \end{pmatrix}.
%$$
%That is, we will show that
%compute $L(\cdots 3 | 4 \cdots) = [3 \cdots] + [4\cdots]$.
%We will show that ${}^{\infty}w|u^3P$ is an extremal section of $A_P$. 
%Let $R^*3 (42)^k | (42)^\ell 3 S$ be a section of $A_P$. 
%Then we have 
%\begin{align*}
%L(R^*3 (42)^k | (42)^\ell 3 S) 
%&= [(24)^k 3 R] + [(42)^\ell 3 S] \\
%&< [ 3 (24)^\infty] + [ (42)^\infty] 
%= \frac{4}{\sqrt{3}}
%\end{align*}
%for $k \ge 1, \ell \ge 0$.
%Since $R$ begins with $2$,
%\begin{equation*}
%L(R^*3 | (42)^\ell 3 S)=[3 R]+[ (42)^\ell 3 S] > [R]+[3 (42)^\ell 3 S]=L(R^* | 3 (42)^\ell 3 S).
%\end{equation*}
Therefore,
$$
\limsup_{R^*|S} L( R^*|S) = 
\limsup_{(R')^*3|S'} L( (R')^*3|S').
$$
Similarly, for $(A_P)^\vee$ we have
$$
\limsup_{R^*|S} L( (R^\vee)^*|S^\vee) = 
\limsup_{(R')^*|3S'} L( (R'^\vee)^*|3S'^\vee).
$$
Hence, we have
\begin{align*}
\mathcal L( A_P) &= \max \left\{ \limsup_{R^*3|S} L( R^*3|S), \,\limsup_{R^*|3S} L( (R^\vee)^*|3S^\vee) \right\}, % \\
%&= \max \left\{ \limsup_{R^*3|S} ( [3R] + [S] ), \, \limsup_{R^*|3S} ( [R^\vee] + [3S^\vee] ) \right\},
\end{align*}
where $R^*3|S$ and $R^*|3S$ runs over all sections of $A_P$. 
Using the fact that for any infinite sequences $Q,R$,
$$
[u Q] > [w R], \quad 
[u^* Q] < [w^* R],
$$
we conclude that 
\begin{align*}
\limsup_{R^*3|S} L( R^*3|S)
&= \limsup_{k \to \infty} L( \cdots w^{k-1} u^3 w v_{k-1} w^{k} | u^3 w v_k w^{k+1} u^3 w v_{k+1} \cdots ) \\
&= L( {}^\infty w | u^3 w P) = [(w^*)^\infty] + [u^3 w P]. 
\end{align*}
We note that $3u^\vee = u^*3$ and $3w^\vee = w^*3$. 
Therefore, we have
\begin{align*}
\limsup_{R^*|3S} L( (R^\vee)^*|3S^\vee) 
&= \lim_{k \to \infty} 
L ( \cdots (u^*)^{n_{k-1}} (w^*)^{k} (u^*)^3| w^* (w^*)^{m_1} (u^*)^{n_1} \cdots ) \\
&\le [u^3 w^\infty] + [(w^*)^\infty]
< [u^3 w P] + [(w^*)^\infty].
\end{align*}
Hence, we have 
\begin{equation*}
\mathcal L( A_P) 
= L( {}^\infty w | u^3 wP) = [(w^*)^\infty] + [u^3 wP]. \qedhere
\end{equation*}
\end{proof}
Let  
$$
K = \left\{ [(w^*)^\infty] + [u^3 w P] \, | \,  P \in E \right\}.
$$
Then, Lemma~\ref{sections} and \eqref{ebound} yield that 
\begin{equation}
K \subset \mathscr L(\mathbf H_6) \cap  \left(0, \frac{4}{\sqrt{3}} + \varepsilon\right).
\end{equation}
Since $[P] \mapsto [u^3 w P ] = N_{uuuw} \cdot [P]$ is a bi-Lipschitz function on $F$, 
we have $\dim_H(K) = \dim_H(F)$
(see \cite{Fal03}*{Corollary 2.4}).
Therefore, Lemma~\ref{dimension} implies  $\dim_H(K) >0$ and we complete the proof of  Theorem~\ref{thm_dimension}.

\section*{Acknowledgement}
The authors would like to thank the anonymous referees for their careful reading and helpful comments.
D.K.~was supported by the National Research Foundation of Korea (NRF) (NRF-2018R1A2B6001624, RS-2023-00245719).
D.S.~was supported by the National Research Foundation of Korea (NRF) (NRF-2018R1A2B6001624, NRF-2020R1A2C1A01011543).
B.C.~was supported by 2025 Hongik University Research Fund.

%\bibliography{reference}

\begin{bibdiv}
\begin{biblist}

%\bib{AA13}{article}{
%  author={Abe, Ryuji},
%  author={Aitchison, Iain R.},
%  title={Geometry and Markoff's spectrum for $\mathbb {Q}(i)$, I},
%  journal={Trans. Amer. Math. Soc.},
%  volume={365},
%  date={2013},
%  number={11},
%  pages={6065--6102},
%  issn={0002-9947},
%  review={\MR {3091276}},
%  doi={10.1090/S0002-9947-2013-05850-3},
%}

%\bib{AAR16}{article}{
%  author={Abe, Ryuji},
%  author={Aitchison, Iain R.},
%  author={Rittaud, Beno\^\i t},
%  title={Two-color Markoff graph and minimal forms},
%  journal={Int. J. Number Theory},
%  volume={12},
%  date={2016},
%  number={4},
%  pages={1093--1122},
%  issn={1793-0421},
%  review={\MR {3484300}},
%  doi={10.1142/S1793042116500676},
%}
%\bib{AR}{article}{
%  author={Abe, Ryuji},
%  author={Rittaud, Beno\^\i t},
%  title={Combinatorics on words associated to Markoff spectra},
%  note={preprint},
%}

\bib{Aig13}{book}{
   author={Aigner, Martin},
   title={Markov's theorem and 100 years of the uniqueness conjecture},
   note={A mathematical journey from irrational numbers to perfect
   matchings},
   publisher={Springer, Cham},
   date={2013},
   pages={x+257},
   isbn={978-3-319-00887-5},
   isbn={978-3-319-00888-2},
   review={\MR{3098784}},
   doi={10.1007/978-3-319-00888-2},
}

\bib{AMU20}{article}{
   author={Artigiani, Mauro},
   author={Marchese, Luca},
   author={Ulcigrai, Corinna},
   title={Persistent Hall rays for Lagrange spectra at cusps of Riemann
   surfaces},
   journal={Ergodic Theory Dynam. Systems},
   volume={40},
   date={2020},
   number={8},
   pages={2017--2072},
   issn={0143-3857},
   review={\MR{4120771}},
   doi={10.1017/etds.2018.143},
}

%\bib{Bar63}{article}{
%   author={Barning, F. J. M.},
%   title={On Pythagorean and quasi-Pythagorean triangles and a generation
%   process with the help of unimodular matrices},
%   language={Dutch},
%   journal={Math. Centrum Amsterdam Afd. Zuivere Wisk.},
%   volume={1963},
%   date={1963},
%   number={ZW-011},
%   pages={37},
%   review={\MR{0190077}},
%}

%\bib{Ber34}{article}{
%author={Berggren, B.},
%title={Pytagoreiska triangular},
%journal={Tidskrift f\"or element\"ar matematik, fysik och kemi},
%volume={17},
%date={1934},
%pages={129--139},
%}

\bib{Bom07}{article}{
   author={Bombieri, Enrico},
   title={Continued fractions and the Markoff tree},
   journal={Expo. Math.},
   volume={25},
   date={2007},
   number={3},
   pages={187--213},
   issn={0723-0869},
   review={\MR{2345177}},
   doi={10.1016/j.exmath.2006.10.002},
}

\bib{CCGW22}{article}{
   author={Cha, Byungchul},
   author={Chapman, Heather},
   author={Gelb, Brittany},
   author={Weiss, Chooka},
   title={Lagrange spectrum of a circle over the Eisensteinian field},
   journal={Monatsh. Math.},
   volume={197},
   date={2022},
   number={1},
   pages={1--55},
   issn={0026-9255},
   review={\MR{4368629}},
   doi={10.1007/s00605-021-01649-y},
}

\bib{CK22}{article}{
   author={Cha, Byungchul},
   author={Kim, Dong Han},
   title={Intrinsic Diophantine approximation on the unit circle and its
   lagrange spectrum},
   language={English, with English and French summaries},
   journal={Ann. Inst. Fourier (Grenoble)},
   volume={73},
   date={2023},
   number={1},
   pages={101--161},
   issn={0373-0956},
   review={\MR{4588926}},
}

%\bib{CK20}{article}{
%   author={Cha, Byungchul},
%   author={Kim, Dong Han},
%   title={Number theoretical properties of Romik's dynamical system},
%   journal={Bull. Korean Math. Soc.},
%   date={2020},
%   volume={57},
%   pages={251--274},
%   doi={/10.4134/BKMS.b190163},
%}

\bib{CNT18}{article}{
   author={Cha, Byungchul},
   author={Nguyen, Emily},
   author={Tauber, Brandon},
   title={Quadratic forms and their Berggren trees},
   journal={J. Number Theory},
   volume={185},
   date={2018},
   pages={218--256},
   issn={0022-314X},
   review={\MR{3734349}},
   doi={10.1016/j.jnt.2017.09.003},
}

\bib{Cus74}{article}{
   author={Cusick, T. W.},
   title={The largest gaps in the lower Markoff spectrum},
   journal={Duke Math. J.},
   volume={41},
   date={1974},
   pages={453--463},
   issn={0012-7094},
   review={\MR{0466018}},
}
\bib{CF89}{book}{
   author={Cusick, Thomas W.},
   author={Flahive, Mary E.},
   title={The Markoff and Lagrange spectra},
   series={Mathematical Surveys and Monographs},
   volume={30},
   publisher={American Mathematical Society, Providence, RI},
   date={1989},
   pages={x+97},
   isbn={0-8218-1531-8},
   review={\MR{1010419}},
   doi={10.1090/surv/030},
}

\bib{Fal03}{book}{
   author={Falconer, Kenneth},
   title={Fractal geometry},
   edition={2},
   note={Mathematical foundations and applications},
   publisher={John Wiley \& Sons, Inc., Hoboken, NJ},
   date={2003},
   pages={xxviii+337},
   isbn={0-470-84861-8},
   review={\MR{2118797}},
   doi={10.1002/0470013850},
}

%\bib{Gbu76}{article}{
%  author={Gbur, Mary E.},
%  title={On the lower Markoff spectrum},
%  journal={Monatshefte f{\"u}r Mathematik},
%  volume={81},
%  date={1976},
%  number={2},
%  pages={95--107},
%  publisher={Springer}
%}

\bib{HS86}{article}{
   author={Haas, Andrew},
   author={Series, Caroline},
   title={The Hurwitz constant and Diophantine approximation on Hecke
   groups},
   journal={J. London Math. Soc. (2)},
   volume={34},
   date={1986},
   number={2},
   pages={219--234},
   issn={0024-6107},
   review={\MR{856507}},
   doi={10.1112/jlms/s2-34.2.219},
}

%\bib{HMTY}{article}{
%   author={Hanson, Elise},
%   author={Merberg, Adam},
%   author={Towse, Christopher},
%   author={Yudovina, Elena},
%   title={Generalized continued fractions and orbits under the action of
%   Hecke triangle groups},
%   journal={Acta Arith.},
%   volume={134},
%   date={2008},
%   number={4},
%   pages={337--348},
%   issn={0065-1036},
%   review={\MR{2449157}},
%   doi={10.4064/aa134-4-4},
%}

%\bib{KLL22}{article}{
%   author={Kim, Dong Han},
%   author={Lee, Seul Bee},
%   author={Liao, Lingmin},
%   title={Odd-odd continued fraction algorithm},
%   journal={Monatsh. Math.},
%   volume={198},
%   date={2022},
%   number={2},
%   pages={323--344},
%   issn={0026-9255},
%   review={\MR{4421912}},
%   doi={10.1007/s00605-022-01704-2},
%}

\bib{KS22}{article}{
  author={Kim, Dong Han},
  author={Sim, Deokwon},
  title={The {Markoff} and {Lagrange} spectra on the {Hecke} group H4},
  date={2022},
  eprint={arXiv:2206.05441 [math.NT]},
}


%\bib{Kop80}{article}{
%   author={Kopetzky, Hans G\"{u}nther},
%   title={Rationale Approximationen am Einheitskreis},
%   language={German, with English summary},
%   journal={Monatsh. Math.},
%   volume={89},
%   date={1980},
%   number={4},
%   pages={293--300},
%   issn={0026-9255},
%   review={\MR{587047}},
%   doi={10.1007/BF01659493},
%}

%\bib{Kop85}{article}{
%   author={Kopetzky, Hans G\"{u}nther},
%   title={\"{U}ber das Approximationsspektrum des Einheitskreises},
%   language={German, with English summary},
%   journal={Monatsh. Math.},
%   volume={100},
%   date={1985},
%   number={3},
%   pages={211--213},
%   issn={0026-9255},
%   review={\MR{812612}},
%   doi={10.1007/BF01299268},
%}

%\bib{KL96}{article}{
%   author={Kraaikamp, Cornelis},
%   author={Lopes, Artur},
%   title={The theta group and the continued fraction expansion with even
%   partial quotients},
%   journal={Geom. Dedicata},
%   volume={59},
%   date={1996},
%   number={3},
%   pages={293--333},
%   issn={0046-5755},
%   review={\MR{1371228}},
%   doi={10.1007/BF00181695},
%}

\bib{Leh85}{article}{
   author={Lehner, J.},
   title={Diophantine approximation on Hecke groups},
   journal={Glasgow Math. J.},
   volume={27},
   date={1985},
   pages={117--127},
   issn={0017-0895},
   review={\MR{819833}},
   doi={10.1017/S0017089500006121},
}

\bib{Mal77}{article}{
  author={Maly\v sev, A. V.},
  title={Markoff and Lagrange spectra (a survey of the literature)},
  language={Russian},
  note={Studies in number theory (LOMI), 4},
  journal={Zap. Nau\v cn. Sem. Leningrad. Otdel. Mat. Inst. Steklov. (LOMI)},
  volume={67},
  date={1977},
  pages={5--38, 225},
  review={\MR {0441876}},
}

\bib{Mar79}{article}{
  author={Markoff, A.},
  title={Sur les formes quadratiques binaires ind\'efinies},
  language={French},
  journal={Math. Ann.},
  volume={15},
  date={1879},
  pages={381--409},
  issn={0025-5831},
}

\bib{Mar80}{article}{
  author={Markoff, A.},
  title={Sur les formes quadratiques binaires ind\'efinies. II},
  language={French},
  journal={Math. Ann.},
  volume={17},
  date={1880},
  number={3},
  pages={379--399},
  issn={0025-5831},
  review={\MR {1510073}},
  doi={10.1007/BF01446234},
}

\bib{MM10}{article}{
   author={Mayer, Dieter},
   author={M\"{u}hlenbruch, Tobias},
   title={Nearest $\lambda_q$-multiple fractions},
   conference={
      title={Spectrum and dynamics},
   },
   book={
      series={CRM Proc. Lecture Notes},
      volume={52},
      publisher={Amer. Math. Soc., Providence, RI},
   },
   date={2010},
   pages={147--184},
   review={\MR{2743437}},
   doi={10.1090/crmp/052/09},
}

\bib{Mor18}{article}{
   author={Moreira, Carlos Gustavo},
   title={Geometric properties of the Markoff and Lagrange spectra},
   journal={Ann. of Math. (2)},
   volume={188},
   date={2018},
   number={1},
   pages={145--170},
   issn={0003-486X},
   review={\MR{3815461}},
   doi={10.4007/annals.2018.188.1.3},
}

%\bib{Mos16}{article}{
%   author={Moshchevitin, Nikolay},
%   title={\"Uber die rationalen Punkte auf der Sph\"are},
%   language={German, with German summary},
%   journal={Monatsh. Math.},
%   volume={179},
%   date={2016},
%   number={1},
%   pages={105--112},
%   issn={0026-9255},
%   review={\MR{3439274}},
%   doi={10.1007/s00605-015-0818-4},
%}

\bib{Nak95}{article}{
   author={Nakada, Hitoshi},
   title={Continued fractions, geodesic flows and Ford circles},
   conference={
      title={Algorithms, fractals, and dynamics},
      address={Okayama/Kyoto},
      date={1992},
   },
   book={
      publisher={Plenum, New York},
   },
   isbn={0-306-45127-1},
   date={1995},
   pages={179--191},
   review={\MR{1402490}},
}
%\bib{Pan19}{article}{
%   author={Panti, Giovanni},
%   title={Billiards on Pythagorean triples and their Minkowski functions},
%   journal={Discrete Contin. Dyn. Syst.},
%   volume={40},
%   date={2020},
%   number={7},
%   pages={4341--4378},
%   issn={1078-0947},
%   review={\MR{4097546}},
%   doi={10.3934/dcds.2020183},
%}

\bib{Pan22}{article}{
   author={Panti, Giovanni},
   title={Attractors of dual continued fractions},
   journal={J. Number Theory},
   volume={240},
   date={2022},
   pages={50--73},
   issn={0022-314X},
   review={\MR{4458233}},
   doi={10.1016/j.jnt.2021.12.011},
}

\bib{Par77}{article}{
   author={Parson, L. Alayne},
   title={Normal congruence subgroups of the Hecke groups $G(2\sp{(1/2)})$
   and $G(3\sp{(1/2)})$},
   journal={Pacific J. Math.},
   volume={70},
   date={1977},
   number={2},
   pages={481--487},
   issn={0030-8730},
   review={\MR{0491507}},
}

\bib{Rom08}{article}{
   author={Romik, Dan},
   title={The dynamics of Pythagorean triples},
   journal={Trans. Amer. Math. Soc.},
   volume={360},
   date={2008},
   number={11},
   pages={6045--6064},
   issn={0002-9947},
   review={\MR{2425702 (2009i:37101)}},
   doi={10.1090/S0002-9947-08-04467-X},
}

\bib{Ros54}{article}{
   author={Rosen, David},
   title={A class of continued fractions associated with certain properly
   discontinuous groups},
   journal={Duke Math. J.},
   volume={21},
   date={1954},
   pages={549--563},
   issn={0012-7094},
   review={\MR{65632}},
}

%\bib{Sch75a}{article}{
%   author={Schmidt, Asmus L.},
%   title={Diophantine approximation of complex numbers},
%   journal={Acta Math.},
%   volume={134},
%   date={1975},
%   pages={1--85},
%   issn={0001-5962},
%   review={\MR {0422168}},
%   doi={10.1007/BF02392098},
%}
 
%\bib{Sch75b}{article}{
%   author={Schmidt, Asmus L.},
%   title={On $C$-minimal forms},
%   journal={Math. Ann.},
%   volume={215},
%   date={1975},
%   pages={203--214},
%   issn={0025-5831},
%   review={\MR{376530}},
%   doi={10.1007/BF01343890},
%}

\bib{Sch76}{article}{
  author={Schmidt, Asmus L.},
  title={Minimum of quadratic forms with respect to Fuchsian groups. I},
  journal={J. Reine Angew. Math.},
  volume={286/287},
  date={1976},
  pages={341--368},
  issn={0075-4102},
  review={\MR {0457358}},
  doi={10.1515/crll.1976.286-287.341},
}

\bib{Sch77}{article}{
   author={Schmidt, Asmus L.},
   title={Minimum of quadratic forms with respect to Fuchsian groups. II},
   journal={J. Reine Angew. Math.},
   volume={292},
   date={1977},
   pages={109--114},
   issn={0075-4102},
   review={\MR{457359}},
   doi={10.1515/crll.1977.292.109},
}

\bib{SS95}{article}{
   author={Schmidt, Thomas A.},
   author={Sheingorn, Mark},
   title={Length spectra of the Hecke triangle groups},
   journal={Math. Z.},
   volume={220},
   date={1995},
   number={3},
   pages={369--397},
   issn={0025-5874},
   review={\MR{1362251}},
   doi={10.1007/BF02572621},
}

%\bib{Sch82}{article}{
%  author={Schweiger, Fritz},
%  title={Continued fractions with odd and even partial quotients},
%  journal={Arbeitsber. Math. Inst. Univ. Salzburg},
%  volume={4},
%  date={1982},
%  pages={59--70},
%}


\bib{Ser88}{article}{
   author={Series, Caroline},
   title={The Markoff spectrum in the Hecke group $G_5$},
   journal={Proc. London Math. Soc. (3)},
   volume={57},
   date={1988},
   number={1},
   pages={151--181},
   issn={0024-6115},
   review={\MR{0940433}},
   doi={10.1112/plms/s3-57.1.151},
}

%\bib{SW16a}{article}{
%   author={Short, Ian},
%   author={Walker, Mairi},
%   title={Even-integer continued fractions and the Farey tree},
%   conference={
%      title={Symmetries in graphs, maps, and polytopes},
%   },
%   book={
%      series={Springer Proc. Math. Stat.},
%      volume={159},
%      publisher={Springer, [Cham]},
%   },
%   date={2016},
%   pages={287--300},
%   review={\MR{3516227}},
%   doi={10.1007/978-3-319-30451-9_15},
%}
%\bib{SW16b}{article}{
%   author={Short, Ian},
%   author={Walker, Mairi},
%   title={Geodesic Rosen continued fractions},
%   journal={Q. J. Math.},
%   volume={67},
%   date={2016},
%   number={4},
%   pages={519--549},
%   issn={0033-5606},
%   review={\MR{3609844}},
%   doi={10.1093/qmath/haw025},
%}
\bib{Vul97}{article}{
   author={Vulakh, L. Ya.},
   title={The Markov spectra for triangle groups},
   journal={J. Number Theory},
   volume={67},
   date={1997},
   number={1},
   pages={11--28},
   issn={0022-314X},
   review={\MR{1485425}},
   doi={10.1006/jnth.1997.2181},
}
\bib{Vul00}{article}{
   author={Vulakh, L. Ya.},
   title={The Markov spectra for Fuchsian groups},
   journal={Trans. Amer. Math. Soc.},
   volume={352},
   date={2000},
   number={9},
   pages={4067--4094},
   issn={0002-9947},
   review={\MR{1650046}},
   doi={10.1090/S0002-9947-00-02455-7},
}

\end{biblist}
\end{bibdiv}
%\printbibliography[title={References}]

\Addresses

\end{document}